\documentclass[11pt]{amsart}
\newtheorem{theorem}{Theorem}
\newtheorem{lemma}[theorem]{Lemma}
\newtheorem{corollary}[theorem]{Corollary}

\usepackage{curves,amsmath,amssymb,color,fancybox,exscale,lineno,hyperref}
\usepackage{bm} \usepackage{mathabx} \usepackage{subfig}
\usepackage{array}
\usepackage{tikz-cd}
\usepackage{mathtools}

\numberwithin{equation}{section}
\numberwithin{theorem}{section}

\def\Nedelec{N\'ed\'elec}

\newcommand{\grad}{\ensuremath{\operatorname{grad}}}
\newcommand{\curl}{\ensuremath{\operatorname{curl}}}
\newcommand{\mydiv}{\ensuremath{\operatorname{div}}}
\newcommand{\Tr}{\ensuremath{\operatorname{Tr}}}

\newcommand{\rd}{{\mathrm{d}}}

\newcommand{\polyspace}[2]{\ensuremath{\mathcal{P}_{#1}^{#2}}}

\def\Xint#1{\mathchoice
  {\XXint\displaystyle\textstyle{#1}}%
  {\XXint\textstyle\scriptstyle{#1}}%
  {\XXint\scriptstyle\scriptscriptstyle{#1}}%
  {\XXint\scriptscriptstyle\scriptscriptstyle{#1}}%
\!\int}
\def\XXint#1#2#3{{\setbox0=\hbox{$#1{#2#3}{\int}$ }
\vcenter{\hbox{$#2#3$ }}\kern-.6\wd0}}

\def\dashint{\Xint-}

\title[Unisolvence for quasi-polynomial forms]{On the unisolvence for
the quasi-polynomial spaces of differential forms}

\author[S. Wu]{Shuonan Wu}
\address{School of Mathematical Sciences,
  Peking University, Beijing 100871, China}
\email{snwu@math.pku.edu.cn}

\author[L. T. Zikatanov]{Ludmil T. Zikatanov}
\address{Department of Mathematics, The Pennsylvania State University,
University Park, Pennsylvania, 16802, USA}
\email{ludmil@psu.edu}
\begin{document}

\begin{abstract}
We consider quasi-polynomial spaces of differential forms defined as
weighted (with a positive weight) spaces of differential forms with
polynomial coefficients. We show that the unisolvent set of
functionals for such spaces on a simplex in any spatial dimension is
the same as the set of such functionals used for the polynomial
spaces.  The analysis in the quasi-polynomial spaces, however, is not
standard and requires a novel approach. We are able to prove our results
without the use of Stokes' Theorem, which is the standard tool in
showing the unisolvence of functionals in polynomial spaces of
differential forms. These new results provide tools for studying
exponentially-fitted discretizations stable for general
convection-diffusion problems in Hilbert differential complexes.
\end{abstract}

\maketitle



\section{Introduction}\label{s:introduction}
The Finite Element (FE) Exterior Calculus (EC)~\cite{2010ArnoldFalkWinther-a,2006ArnoldFalkWinther-a} is a powerful technique that combines tools from differential geometry and finite element analysis in constructing discretizations which inherit the natural structure of the underlying physical models. In our work, we consider general convection-diffusion equations on Hilbert complexes, such as the ones involving $H(\grad)$, $H(\curl)$, and $H(\mydiv)$ in 3D,
describing diffusion (by Hodge Laplacian) and corresponding transport driven by different velocity fields.

The design of stable discretizations for convection-diffusion problems, even in the scalar case,  is a challenging task as these are singularly perturbed differential equations with small, and even vanishing, diffusion (see, e.g.~\cite{1996RoosStynesTobiska-a} and the references therein for discussion on such topics). There is a vast amount of literature on various techniques designed to take care of the numerical instabilities associated with this type of equations. We refer the reader to recent and classical works on the subject focused on some of these techniques: mixed FE methods~\cite{2019BurmanHe-a,2006BrezziMariniMichelettiPietraSacco-a,2002HoustonSchwabSueli-a,1996BrezziMariniPietraRusso-a,1989BrezziMariniPietra-a};
discontinuous Galerkin methods~\cite{2014CangianiGeorgoulisMetcalfe-a,2008DiPietroErnGuermond-a,2007GuermondPopov-a,1974LasaintRaviart-a};
discontinuous Petrov-Galerkin methods~\cite{2019BertrandDemkowiczGopalakrishnanHeuer-a,2011DemkowiczGopalakrishnan-a,2010DemkowiczGopalakrishnan-a}.

Our results show unisolvence for the quasi-polynomial (weighted) spaces used in simplex-averaged finite element (SAFE) discretization~\cite{2019WuXu-a} for convection-diffusion equations in Hilbert complexes. Such exponentially fitted finite element schemes have been used with success for scalar convection-diffusion equations, i.e., in our terminology for convection-diffusion problems in $H(\grad)$.
A rough explanation of the ideas behind SAFE discretizations could be as follows: (1) define a variable representing the flux, as in mixed methods, and use a variable change to symmetrize the equation; (2) discretize the differential operator using discretization for the flux and the primal variable; (3) eliminate the flux (locally) and change the variables to obtain a discretization of the original problem.
Such a path for the derivation of discrete problems is seen in the pioneering work on discretizing drift-diffusion models in 1D~\cite{1969ScharfetterGummel-a} and later in FE and finite volume schemes in higher spatial dimensions~\cite{1998BankCoughranCowsar-a,1999XuZikatanov-a,2005LazarovZikatanov-a}. Recently,
a more general FEEC  approach has brought mechanisms that can utilize higher degree polynomials and can work in any spatial dimension. In addition to the SAFE discretizations~\cite{2019WuXu-a}, the FEEC approach was an important tool in designing exponentially fitted space-time discretizations in~\cite{2017BankVassilevskiZikatanov-a}.

In this work we consider one of the key ingredients needed in steps
(2) and (3) above, namely, determining a set of unisolvent functionals
for the numerical flux. A typical situation in the discretizations
discussed above is the following: Given a polynomial vector space
$\mathcal{P}$, on an $n$-dimensional simplex $T$, we discretize the
flux using a quasi-polynomial space of differential forms,
\begin{equation*} \label{eq:eP-general}
e\mathcal{P}\Lambda^k(T) := \{
e(\bm{x}) \omega ~|~ \omega \in \mathcal{P}\Lambda^k(T) \},
\end{equation*}
where $\mathcal{P}\Lambda^k(T)$ denotes the space of $k$-forms in $\mathbb{R}^n$ with coefficients from $\mathcal{P}$.
For example, for a convection-diffusion equation in $H(\grad)$ in 3D the degrees of freedom (unisolvence functionals) which uniquely determine an element $p\in \mathcal{P}$, are the moments of $p$ on edges of $T$, faces of $T$ and $T$ itself.
The classical works~\cite{1988Bossavit-a,1999Hiptmair-a,2006ArnoldFalkWinther-a,2010ArnoldFalkWinther-a}
usually use the Stokes' Theorem, when verifying the unisolvence of
such functionals for $\mathcal{P}$, and the arguments involve
differentiation of $p$. For quasi-polynomial spaces, with a
non-constant weight $e(\bm{x})$, such differentiation results in terms
that have both derivatives of $p$ and derivatives of $e$. The standard
arguments are, therefore, not applicable except in some special cases,
such as the lowest order first kind of
\Nedelec{}-Raviart-Thomas elements. Our analysis here circumvents the
use of Stokes' Theorem, and we are able to show that the unisolvence
functionals for $\mathcal{P}$ are also unisolvence functionals for
$e\mathcal{P}$ for differential forms of all orders $k$, in any
spatial dimension $n$, and all polynomial spaces of the first and second
kind
(\Nedelec{}--Raviart--Thomas~\cite{1977RaviartThomas-a,1986Nedelec-a,1980Nedelec-a},
 \Nedelec{}--Brezzi--Douglas--Marini~\cite{1986Nedelec-a,1980Nedelec-a,1985BrezziDouglasMarini-a}
 spaces).

The landscape of the paper can be mapped as follows: Preliminaries and FEEC notation is introduced in Section~\ref{s:preliminaries}. The unisolvence sets of functionals for the first kind (\Nedelec{}--Raviart--Thomas) and second kind (\Nedelec--Brezzi--Douglas--Marini) are discussed in Section~\ref{s:first-kind} and Section~\ref{s:second-kind}, respectively. Examples for constructing discretizations of the flux are then given in Section~\ref{s:applications}, and numerical tests are shown in Section~\ref{s:numerics}.

\section{Preliminaries}\label{s:preliminaries}

In this section, we present some preliminary results which will be used in the following sections. We begin by a simple result, frequently used in the analysis.
\begin{lemma} \label{lm:no-change-sign}
Let $D$ be an open domain and $f: \bar{D} \to \mathbb{R}$ be a Riemann
integrable function. If $f|_{\bar{D}} \geq 0$ (or $f|_{\bar{D}} \leq
0$), then $\int_D f(x) = 0$ implies that $f|_{\bar{D}} = 0$ almost
everywhere.
\end{lemma}

We denote the spaces of polynomials in $n$ variables of degree at most $r$ and of homogeneous polynomial functions
of degree $r$ by $\mathcal{P}_r(\mathbb{R}^n)$ and $\mathcal{H}_r(\mathbb{R}^n)$,  respectively.
We will abbreviate them to $\mathcal{P}_r$ and $\mathcal{H}_r$ at times. Next, following \cite{2006ArnoldFalkWinther-a}, we present some basic notation commonly used in FEEC when working with polynomial differential forms.

\subsection{Simplices and barycentric coordinates}
Let $\Sigma(k,n)$ denote the set of increasing maps $\{1, \ldots, k\}
\to \{1, \ldots, n\}$, for $1\leq k \leq n$. $\rho^*\in
\Sigma(n-k,n)$ is the complementary map of $\rho \in \Sigma(k,n)$ with
$k<n$.  For any $\rho \in \Sigma(k,n)$, denote $(0, \rho): \{0, 1,
\ldots, k\} \to \{0, 1,\ldots n\}$ by $(0, \rho)(0) = 0$. Similarly,
let $\Sigma_0(k,n)$ denote the set of increasing maps $\{0, \ldots,
k\} \to \{0, \ldots, n\}$ for $0\leq k \leq n$. The map complementary
to $\rho \in \Sigma_0(k,n)$, denoted by $\rho^*$, satisfies $\rho^*
\in \Sigma_0(n-k-1,n)$ such that $\mathcal{R}(\rho) \cup
\mathcal{R}(\rho^*) = \{0,\ldots,n\}$.  Here, $\mathcal{R}(\rho)$
represents the range of $\rho$ in ascending order, which is also
denoted by $\rho$ if there is no ambiguity. In addition, $|\rho|$
denotes the cardinality of $\mathcal{R}(\rho)$.

Let $T := [x_0, x_1, \ldots, x_n]$ be an $n$-simplex with the vertices
$x_i$. For each $\rho \in \Sigma_0(k,n)$, the set $f_\rho :=
[x_{\rho(0)}, \ldots, x_{\rho(k)}]$ is a subsimplex of
dimension $k$. For $k < n$, $f_{\rho^*}$ is the
$(n-k-1)$-dimensional subsimplex of $T$ opposite to the
$k$-subsimplex $f_\rho$. The set of subsimplices of dimension $k$ of
$T$ is denoted by $\Delta_k(T)$, and the set of all subsimplices of
$T$ is denoted by $\Delta(T)$.

We denote by $\lambda_0, \ldots, \lambda_n$ the barycentric
coordinates satisfying $\lambda_i(x_j) = \delta_{ij}$.
Clearly, $\lambda_i$ form a basis of $\mathcal{P}_1(\mathbb{R}^n)$ and
satisfy $\sum_{i} \lambda_i = 1$. For a sub-simplex $f = f_\rho$ with $\rho \in \Sigma_0(k,n)$, there is an
isomorphism between $\mathcal{P}_r(f)$ of polynomial functions on $f$ and
the space $\mathcal{H}_r(\mathbb{R}^{k+1})$. That is, each $p \in \mathcal{P}_r(f)$ can be expressed as
$$
p(x) = q\left(\lambda_{\rho(0)}(x), \ldots, \lambda_{\rho(k)}(x) \right)  \quad x \in f,
$$
for a unique $q \in \mathcal{H}_r(\mathbb{R}^{k+1})$. The extension, denoted by $E_{f,T}(p)$, is defined
by extending the right-hand side for $x \in \mathbb{R}^n$. It is readily seen that the extension $E_{f,T}(p)$ is an injective mapping from $\mathcal{P}_r(f)$ to $\mathcal{P}_r(\mathbb{R}^n)$.

Since the vectors $t_i := x_i - x_0~
(i=1,\ldots, n)$ form a basis for $\mathbb{R}^n$, the dual basis
functions $\rd\lambda_i~(i=1,\ldots,n)$ form a basis for
${\rm Alt}^1\mathbb{R}^n$. For any face $f_\rho$, the
restrictions of $\rd\lambda_{\rho(1)}, \ldots,
\rd\lambda_{\rho(k)}$ to the tangent space $V$ of $f_{\rho}$ at
any point of $f_\rho$ give a basis for ${\rm Alt}^1V$.

The algebraic $k$-forms $(\rd\lambda)_\rho := \rd\lambda_{\rho(1)}
\wedge \cdots \wedge \rd\lambda_{\rho(k)}$, $\rho\in \Sigma(k,n)$, form a
basis for ${\rm Alt}^k$. Hence, a differential $k$-form $\omega$
can be uniquely written in the form
\[
\omega = \sum_{\sigma(k,n)} a_\sigma(\rd\lambda)_\sigma.
\]
By definition of wedge product, we have $\rd\lambda_1 \wedge \cdots
\wedge \rd\lambda_n(t_1, \cdots, t_n) = 1$, which implies that
\begin{equation} \label{eq:volumn-form}
\rd\lambda_1 \wedge \cdots \wedge \rd\lambda_n = \frac{1}{n!
  |T|}\operatorname{vol}_T,
\end{equation}
where $\operatorname{vol}_T$ denotes the volume form in
$\Lambda^n(T)$.

\subsection{Whitney forms} For any $\rho \in \Sigma_0(k,n)$ and $f := f_\rho \in \Delta_k(T)$, an
associated differential $k$-form (called {\it Whitney form}) is given
by
\begin{equation} \label{eq:Whitney}
\phi_\rho := \sum_{i=0}^k (-1)^i \lambda_{\rho(i)} \rd
\lambda_{\rho(0)} \wedge \cdots \wedge \widecheck{\rd
  \lambda_{\rho(i)}} \wedge \cdots \wedge \rd \lambda_{\rho(k)},
\end{equation}
where the inverted hat represents a suppressed argument. As shown
below, the Whitney form gives an explicit formulation of the basis of
$\mathcal{P}_1^-\Lambda^k(T)$.  \begin{theorem}[Theorem~4.1 in
\cite{2006ArnoldFalkWinther-a}] \label{tm:Whitney} The Whitney
$k$-forms $\phi_{\rho}$ corresponding to $f_\rho \in \Delta_k(T)$ form
a basis for $\mathcal{P}_1^-\Lambda^k(T)$.
\end{theorem}

\section{Unisolvence for quasi-polynomial spaces of first kind}
\label{s:first-kind}
In this section we consider the first type quasi-polynomial
\begin{equation} \label{eq:eP-minus}
e\mathcal{P}_r^-\Lambda^k(T) := \{
e(\bm{x}) \omega ~|~ \omega \in \mathcal{P}_r^-\Lambda^k(T) \},
\end{equation}
Here, $e(\bm{x}) > 0$ is a general positive weight on $\bar{T}$.

\subsection{Geometrical decomposition of the first kind polynomial spaces}
We now introduce the geometrical decomposition of
$\mathcal{P}_r^-\Lambda^k(T)^*$ and $\mathcal{P}_r^-\Lambda^k(T)$. The
degrees of freedom of $\mathcal{P}_r^-\Lambda^k(T)$ as given in \cite[Section~4.6]{2006ArnoldFalkWinther-a} are
\begin{equation}\label{eq:dof-minus}
\int_f \Tr_f \omega \wedge \eta, \quad \eta \in \mathcal{P}_{r+k-\dim
f-1}\Lambda^{\dim f - k}(f), \quad f \in \Delta(T).
\end{equation}
The proof (see, e.g.~\cite{2006ArnoldFalkWinther-a})
that these functionals form a unisolvent set uses induction argument and the
Stokes' Theorem, and the arguments do not carry over to quasi-polynomial spaces.
An attempt to prove the result for quasi-polynomial spaces~\eqref{eq:eP-minus}, however, reveals that
the characterization of the trace free part of the $\mathcal{P}_r^-\Lambda^k(T)$
plays a  crucial role in showing the unisolvence.
Such a characterization is given in the theorem below.
\begin{theorem}[Theorem 4.16 in \cite{2006ArnoldFalkWinther-a}]
\label{tm:trace-free-minus}
For $1\leq k \leq n$, $r \geq n+1 - k$, the map
\begin{equation}\label{eq:trace-free-minus}
\sum_{\rho \in \Sigma(k,n)} a_\rho (\rd \lambda)_{\rho^*} \mapsto
\sum_{\rho \in \Sigma(k,n)} a_\rho \lambda_{\rho^*} \phi_{(0,\rho)},
\end{equation}
where the $a_\rho \in \mathcal{P}_{r+k-n-1}(T)$, defines an
isomorphism of $\mathcal{P}_{r+k-n-1}\Lambda^{n-k}(T)$ onto
$\mathring{\mathcal{P}}_r^-\Lambda^k(T)$. Here, $\lambda_{\rho^*} := \prod_{i=1}^{n-k} \lambda_{\rho^*(i)}$.
\end{theorem}
Next, we use this Theorem to show that the unisolvence functionals for the polynomial space also work for the quasi-polynomials.
\subsection{Polynomials of first kind with vanishing traces}
The main result in this section is the following lemma.
\begin{lemma}\label{lm:quasi-minus}
Let $\omega \in \mathring{\mathcal{P}}_r^- \Lambda^k(T)$. Suppose that
\begin{equation}\label{eq:eP-minus-dofT}
\int e(\bm{x}) \omega \wedge \eta = 0, \qquad \eta \in
\mathcal{P}_{r-n+k-1}\Lambda^{n-k}(T).
\end{equation}
Then $\omega = 0$.
\end{lemma}
Postponing the proof of this lemma for later, we note that its implications show the desired unisolvence results. Indeed,
by induction argument (from low dimension sub-simplices to high dimension sub-simplices), it can be easily shown that
the functionals given in~\eqref{eq:dof-minus} also give degrees of freedom for
$e\mathcal{P}_r^-\Lambda^k(T)$, and we have the following theorem.

\begin{theorem} \label{tm:quasi-minus}
Let $0 \leq k \leq n$, $r \geq 1$. Suppose that $\omega \in
\mathcal{P}_r^-\Lambda^k(T)$
satisfies
$$
\int_f \mathrm{Tr}_f (e(\bm{x}) \omega) \wedge \eta = 0, \quad \eta
\in \mathcal{P}_{r+k-\mathrm{dim}f-1}\Lambda^{\mathrm{dim}f-k}(f),
\quad f \in \Delta(T).
$$
Then $\omega = 0$.
\end{theorem}
\begin{proof}
For any $f \in \Delta_k(T)$, the trace of $\omega \in \mathcal{P}_r^-\Lambda^k(T)$ on $\partial f$ vanishes, as it is a $k$-form on a manifold of dimension $k-1$. Noting that $\mathrm{Tr}_f (e(\bm{x})\omega) = e(\bm{x})\mathrm{Tr}_f\omega$,
applying $T = f$ in Lemma \ref{lm:quasi-minus}, we have $\mathrm{Tr}_f
\omega = 0$. Next, for any $f\in \Delta_{k+1}(T)$,
$$
\int_f \mathrm{Tr}_f(e(\bm{x})\omega) \wedge \eta = 0, \quad \eta \in
\mathcal{P}_{r-2}\Lambda^1(f), \quad f\in \Delta_{k+1}(T),
$$
which implies that $\mathrm{Tr}_f \omega = 0$ by applying Lemma~\ref{lm:quasi-minus} again.
The proof is completed by an induction argument.
\end{proof}

Before we proceed the proof of Lemma \ref{lm:quasi-minus} for general
case, we first give some examples to fix the ideas as abstractions can
often be difficult to grasp.  Noting that cases for $k$-forms in which
$k = 0$ or $k = n$ are trivial.

\begin{proof}[Proof of Lemma~\ref{lm:quasi-minus} for 1-forms in $2D$] In this case, there are two
maps in $\Sigma(1,2)$, namely
\[
\rho_1(1) = 1, \quad
\rho_2(1) = 2.
\]
In light of \eqref{eq:trace-free-minus}, $\omega$ can be uniquely
written as
\[
\begin{aligned}
\omega &= a_1 \lambda_2 \phi_{(0,\rho_1)} + a_2 \lambda_1\phi_{(0,
\rho_2)} \\
&= a_1 \lambda_2(\lambda_0 \rd \lambda_1 - \lambda_1\rd \lambda_0) +
a_2\lambda_1(\lambda_0\rd \lambda_2 - \lambda_2\rd\lambda_0),
\end{aligned}
\]
where $a_i \in \mathcal{P}_{r-2}(T), i=1,2$. We choose a special test
form $\eta$ defined as
\[
\eta = a_1 \rd \lambda_2 - a_2 \rd \lambda_1 \in
\mathcal{P}_{r-2}\Lambda^1(T).
\]
We note here that the sign of $a_2$ is in accordance with the isomorphism
defined in \eqref{eq:trace-free-minus}. Using that $\lambda_0 + \lambda_1 +
\lambda_2 = 1$, $\rd\lambda_0 = -\rd\lambda_1 - \rd\lambda_2$,
and collecting the coefficients of $\omega \wedge \eta$ then shows that
\[
\begin{aligned}
0 &= \int e(\bm{x}) \omega \wedge \eta \\
& = \int e(\bm{x}) [a_1, a_2]
\begin{bmatrix}
\lambda_2(\lambda_0 + \lambda_1) & \lambda_1\lambda_2 \\
\lambda_1\lambda_2 & \lambda_1(\lambda_0 + \lambda_2)
\end{bmatrix}
\begin{bmatrix}
a_1 \\
a_2
\end{bmatrix}
\rd\lambda_1\wedge \rd\lambda_2.
\end{aligned}
\]
On the other hand the polynomial function under the integral is non-negative as seen below,
\[
\begin{aligned}
& [a_1, a_2]
\begin{bmatrix}
\lambda_2(\lambda_0 + \lambda_1) & \lambda_1\lambda_2 \\
\lambda_1\lambda_2 & \lambda_1(\lambda_0 + \lambda_2)
\end{bmatrix}
\begin{bmatrix}
a_1 \\
a_2
\end{bmatrix} \\
= & \lambda_0(a_1^2 \lambda_1 + a_2^2\lambda_2) +
\lambda_1\lambda_2(a_1 + a_2)^2 \geq 0.
\end{aligned}
\]
Finally, using that the weight is positive, i.e.,
$e(\bm{x}) > 0$ together with Lemma \ref{lm:no-change-sign}, shows that $a_1 = a_2 = 0$ and therefore proves Lemma~\ref{lm:quasi-minus} for $n=2$ and $k=1$.
\end{proof}

\begin{proof}[Proof of Lemma~\ref{lm:quasi-minus} for $1$-forms and any spatial dimension $n$] Consider the general case $\omega \in
\mathring{\mathcal P}_r^-\Lambda^1(T)$. In a similar way, there are
$n$ maps in $\Sigma(1,n)$, denoted by $\{\rho_i\}_{i=1}^n$, where
$\rho_i(1) = i$. Hence, $\omega$ can be uniquely written as
$$
\omega = \sum_{i=1}^n a_i \left(\prod_{j=1,j\neq i}^n\lambda_j \right)
\phi_{(0,\rho_i)}
 = \sum_{i=1}^n a_i \left(\prod_{j=1,j\neq i}^n\lambda_j \right)
(\lambda_0 \rd \lambda_i - \lambda_i \rd \lambda_0),
$$
where $a_i \in \mathcal{P}_{r-n}(T)$.  Taking a special test
$(n-1)$-form as
$$
\eta = \sum_{i=1}^n (-1)^{i+1} a_i (\rd
    \lambda)_{\rho_i^*} = \sum_{i=1}^n (-1)^{i+1}a_i \rd \lambda_1
\wedge \cdots \wedge \widecheck{\rd \lambda_i} \wedge \cdots \wedge
\rd\lambda_n,
$$
which gives
$$
\omega \wedge \eta = \bm{a}^T M \bm{a}
\frac{\operatorname{vol}_T}{n!|T|}, \qquad \text{where }\bm{a} = [a_1,
\cdots, a_n]^T,
$$
and $M = (m_{ij})$ with
\begin{equation} \label{eq:M-1}
m_{ij} = \left\{
\begin{aligned}
(\lambda_0 + \lambda_i) \prod_{l=1,l\neq i}^n\lambda_l & \quad i=j,\\
\prod_{l=1}^n \lambda_l & \quad i\neq j.
\end{aligned}
\right.
\end{equation}
If we now denote \(b_i=\prod_{l=0,l\neq i}^n\lambda_l\),
$i=0,\ldots,n$, we have
$$
M = \operatorname{diag}(b_1,\ldots,b_n) +
b_0\bm{1}\bm{1}^T,\quad \bm{1}=(1,\ldots,1)^T,
$$
or
\begin{equation} \label{eq:atma1}
\bm{a}^TM\bm{a} =  \left(\sum_{i=1}^n a_i^2 b_i\right) + b_0\left(
\sum_{i=1}^n a_i \right)^2.
\end{equation}
Since $b_i > 0$, $i=0,\ldots,n$ for $\bm{x}$ in the interior of $T$, this shows
Lemma~\ref{lm:quasi-minus} for $k=1$ and any spatial dimension $n$.
\end{proof}

\begin{proof}[Proof  of Lemma~\ref{lm:quasi-minus} for $(n-1)$-forms in any spatial dimension $n$] Consider the case in which $\omega
\in \mathring{\mathcal{P}}_r^- \Lambda^{n-1}(T)$. Again, there are
$n$ maps in $\Sigma(n-1, n)$, whose complements we  denote by
$\rho_i^*$ with $\rho_i^*(1) = i$. Next, we rewrite $\omega$ as
\[
\omega = \sum_{i=1}^n a_i^* \lambda_i \phi_{(0,\rho_i)},
\]
where $a_i^* \in \mathcal{P}_{r-2}(T)$. Choose a special test function
\[
\eta = \sum_{i=1}^n (-1)^{n+i} a_i^* \rd \lambda_i.
\]
Then, we have $\omega \wedge \eta = (\bm{a}^*)^T M \bm{a}^*
\frac{\operatorname{vol}_T}{n!|T|}$, where $\bm{a}^* = [a_1^*, \ldots,
a_n^*]^T$ and $M = (m_{ij})$ with
$$
m_{ij} = \left\{
\begin{aligned}
\lambda_i (1-\lambda_i) & \quad i=j,\\
(-1)^{i+j+1}\lambda_i\lambda_j & \quad i\neq j.
\end{aligned}
\right.
$$
%
We then have
\begin{equation}\label{eq:atma2}
(\bm{a}^*)^T M \bm{a}^* = \lambda_0 \sum_{i=1}^n (a_i^*)^2 \lambda_i
+ \sum_{1\leq i <  j \leq n}\lambda_i \lambda_j \left(
(-1)^ia_i^* - (-1)^j
    a_j^*
\right)^2 \geq 0.
\end{equation}
Therefore, we see that Lemma~\ref{lm:quasi-minus} holds for any  $n$ with $k=(n-1)$.
\end{proof}

\subsection{Summary (spacial cases  of Lemma~\ref{lm:quasi-minus})}
Let us summarize what we have shown so far: Lemma \ref{lm:quasi-minus}
holds for any $k$-form for $k = 0, 1, n-1, n$. As a consequence, for spatial dimensions $n \leq 3$, we have
proved Lemma~\ref{lm:quasi-minus} in all the possible cases.

To generalize the ideas for other values of $n$ and $k$, we proceed as
in the special cases considered above. For a given $\omega \in
\mathring{\mathcal P}_r^-\Lambda^k(T)$, we find a special test
form $\eta \in \mathcal{P}_{r-n+k-1}\Lambda^{n-k}(T)$ so that $\omega
\wedge \eta$ does not change sign on $\bar{T}$.
This gives us a ``mass'' matrix $M$ which corresponds to the Whitney form bases in $\mathring{\mathcal P}_r^-\Lambda^k(T)$ and
$\mathcal{P}_{r-n+k-1}\Lambda^{n-k}(T)$.

We now follow this plan and generalize the unisolvence result to discrete differential forms of arbitrary order in any spatial dimension and any quasi-polynomial Hilbert complex of first kind.

\subsection{Calculating the ``mass'' matrix}
Let us fix the spatial dimension $n$ and recall that
$$
\frac{1}{n!|T|}\operatorname{vol}_T = \rd \lambda_1 \wedge \cdots \wedge
\rd\lambda_n.
$$
We now define $\sigma(\rho)$ as the {\it number of inversions} of the
array corresponding to $\mathcal{R}(\rho) \mathcal{R}(\rho^*)$. For
instance, when $\rho \in \Sigma(2,5)$ such that $\mathcal{R}(\rho) =
\{3,5\}$, then $\mathcal{R}(\rho)\mathcal{R}(\rho^*)$ associates
with the array $3,5,1,2,4$ and hence $\sigma(\rho) = 5$. It is easy to
show that $\sigma(\rho^*) = k(n-k) - \sigma(\rho)$ for any $\rho \in
\Sigma(k,n)$.

We first give the following result relating the maps $(0,\rho)$,
   $\rho^*$ and the Whitney forms.
\begin{lemma} \label{lm:equal-minus}
For any $\rho \in \Sigma(k,n)$,
\begin{equation} \label{eq:equal-minus}
\phi_{(0,\rho)} \wedge (\rd\lambda)_{\rho^*} =
(-1)^{\sigma(\rho)}\left(
\lambda_0 + \sum_{i=1}^k \lambda_{\rho(i)}
\right) \frac{\operatorname{vol}_T}{n!|T|}.
\end{equation}
\end{lemma}
\begin{proof}
We use the definition of the Whitney form \eqref{eq:Whitney}, to obtain that
\begin{equation} \label{eq:Whitney2}
\phi_{(0,\rho)} = \lambda_0 (\rd\lambda)_{\rho} + \sum_{i=1}^k (-1)^i
\lambda_{\rho(i)} \rd\lambda_0 \wedge \cdots \wedge
\widecheck{\rd\lambda_{\rho(i)}} \wedge \cdots \wedge
\rd\lambda_{\rho(k)}.
\end{equation}
Note that for any $1\leq i \leq k$,
$$
\begin{aligned}
& (-1)^i \lambda_{\rho(i)} \rd\lambda_0 \wedge \cdots \wedge
\widecheck{\rd\lambda_{\rho(i)}} \wedge \cdots \wedge
\rd\lambda_{\rho(k)} \wedge (\rd\lambda)_{\rho^*} \\
=~ & (-1)^{i-1} \lambda_{\rho(i)}
\rd\lambda_{\rho(i)} \wedge \cdots \wedge
\widecheck{\rd\lambda_{\rho(i)}} \wedge \cdots \wedge
\rd\lambda_{\rho(k)} \wedge (\rd\lambda)_{\rho^*} \\
=~ & \lambda_{\rho(i)} (\rd\lambda)_{\rho} \wedge
(\rd\lambda)_{\rho^*} = (-1)^{\sigma(\rho)}
\lambda_{\rho(i)} \frac{\operatorname{vol}_T}{n!|T|}.
\end{aligned}
$$
The result follows by summing up the identities above.
\end{proof}

\begin{lemma} \label{lm:cross-minus}
For $\rho, \widetilde{\rho} \in \Sigma(k,n)$ and $\rho \neq \widetilde{\rho}$,
it holds that
\begin{equation} \label{eq:cross-minus}
\begin{aligned}
& ~\quad \phi_{(0,\rho)} \wedge (\rd\lambda)_{\widetilde{\rho}^*} \\
& =
\left\{
\begin{aligned}
(-1)^{\sigma(\widetilde{\rho}) + s + t} \lambda_{\rho\cap \widetilde{\rho}^*}
\frac{\operatorname{vol}_T}{n!|T|} &
\quad \text{if } \rho \cap
\widetilde{\rho}^*
= \{\rho(s)\}, \rho^* \cap
\widetilde{\rho} = \{\widetilde{\rho}(t)\}, \\
0 & \quad \text{otherwise}.
\end{aligned}
\right.
\end{aligned}
\end{equation}
\end{lemma}
\begin{proof}
Since $|\rho| = |\widetilde{\rho}| = k$, $\rho \neq \widetilde{\rho}$,
we easily see that $|\rho \cap \widetilde{\rho}^*| \geq 1$. Moreover,
if $|\rho \cap \widetilde{\rho}^*| \geq 2$, we deduce $\phi_{(0,\rho)}
\wedge (\rd \lambda)_{\widetilde{\rho}^*} = 0$ from equation
\eqref{eq:Whitney2}.

For the case in which $|\rho \cap \widetilde{\rho}^*| = 1$, we see that
$|\rho^* \cap \widetilde{\rho}| = 1$. Notice that $\rho\cap \widetilde{\rho}^*
= \{\rho(s)\}$ and $\rho^* \cap \widetilde{\rho} = \{\widetilde{\rho}(t)\}$, then
$$
\begin{aligned}
\phi_{(0,\rho)} \wedge (\rd\lambda)_{\widetilde{\rho}^*} & = (-1)^s
\lambda_{\rho\cap \widetilde{\rho}^*} \rd \lambda_0 \wedge \cdots \wedge
\widecheck{\rd \lambda_{\rho(s)}} \wedge \cdots \wedge
\rd\lambda_{\rho(k)} \wedge (\rd \lambda)_{\widetilde{\rho}^*} \\
&= (-1)^{s-1}
\lambda_{\rho\cap \widetilde{\rho}^*} \rd \lambda_{\rho^*\cap
  \widetilde{\rho}} \wedge \cdots \wedge \widecheck{\rd \lambda_{\rho(s)}}
  \wedge \cdots \wedge \rd\lambda_{\rho(k)} \wedge (\rd
      \lambda)_{\widetilde{\rho}^*} \\
&= (-1)^{s-1}
\lambda_{\rho\cap \widetilde{\rho}^*} \rd \lambda_{\widetilde{\rho}(t)} \wedge
\cdots \wedge \widecheck{\rd \lambda_{\rho(s)}} \wedge \cdots \wedge
\rd\lambda_{\rho(k)} \wedge (\rd \lambda)_{\widetilde{\rho}^*} \\
&= (-1)^{s+t}
\lambda_{\rho\cap \widetilde{\rho}^*}
(\rd \lambda)_{\widetilde{\rho}}\wedge (\rd \lambda)_{\widetilde{\rho}^*} \\
& = (-1)^{\sigma(\widetilde{\rho}) + s + t} \lambda_{\rho\cap
\widetilde{\rho}^*} \frac{\operatorname{vol}_T}{n!|T|}.
\end{aligned}
$$
This completes the proof.
\end{proof}
\begin{corollary} \label{cor:symmetry-minus}
For any $\rho, \widetilde{\rho} \in \Sigma(k,n)$, it holds that
\begin{equation}\label{eq:symmetry-minus}
\lambda_{\rho^*} \phi_{(0,\rho)} \wedge (-1)^{\sigma(\widetilde{\rho})}
(\rd \lambda)_{\widetilde{\rho}^*} =
\lambda_{\widetilde{\rho}^*}\phi_{(0,\widetilde{\rho})}
\wedge (-1)^{\sigma(\rho)} (\rd \lambda)_{\rho^*}.
\end{equation}
\end{corollary}
\begin{proof}
We verify the statement case by case:
\begin{itemize}
\item $\rho = \widetilde{\rho}$: obvious.
\item $|\rho \cap \widetilde{\rho}^*| \geq 2$: both left hand side (LHS) and right hand side (RHS) are zero.
\item $|\rho \cap \widetilde{\rho}^*| = 1$: Recall that $\rho \cap \widetilde{\rho}^* =
\{\rho(s)\}, \rho^* \cap \widetilde{\rho} = \{\widetilde{\rho}(t)\}$, by
\eqref{eq:cross-minus},
$$
\text{LHS } = (-1)^{s+t}\lambda_{\rho^*} \lambda_{\rho \cap
\widetilde{\rho}^*} \frac{\operatorname{vol}_T}{n!|T|}
= (-1)^{s+t}\lambda_{\rho^* \cup \widetilde{\rho}^*}
\frac{\operatorname{vol}_T}{n!|T|} =
\text{RHS}.
$$
\end{itemize}
This completes the proof.
\end{proof}

\subsection{Proof of Lemma~\ref{lm:quasi-minus}}
Let us consider a differential form  $\omega \in \mathring{\mathcal P}_{r}^-\Lambda^k(T)$. By
Theorem \ref{tm:trace-free-minus}, we can write $\omega$ as
$$
\omega = \sum_{\rho\in \Sigma(k,n)} a_\rho \lambda_{\rho^*}
\phi_{(0,\rho)},
$$
where $a_\rho \in \mathcal{P}_{r+k-n-1}(T)$.  We take a special test form $\eta$
in \eqref{eq:eP-minus-dofT} as
$$
\eta = \sum_{\rho\in \Sigma(k,n)} a_\rho (-1)^{\sigma(\rho)} (\rd
    \lambda)_{\rho^*} \in \mathcal{P}_{r+k-n-1}\Lambda^{n-k}(T).
$$
By Lemma~\ref{lm:equal-minus} and Lemma~\ref{lm:cross-minus}, we have
\begin{equation} \label{eq:omega-eta-minus}
\omega \wedge \eta = \sum_{\rho, \widetilde{\rho} \in \Sigma(k,n)} a_\rho
a_{\widetilde{\rho}} m_{\rho \widetilde{\rho}}
\frac{\operatorname{vol}_T}{n!|T|},
\end{equation}
where $M = (m_{\rho \widetilde{\rho}})$ is symmetric (by Corollary
\ref{cor:symmetry-minus}) and
\begin{equation} \label{eq:M-minus}
m_{\rho\widetilde{\rho}} =
\left\{
\begin{aligned}
\lambda_{\rho^*} (\lambda_0 + \sum_{i=1}^k\lambda_{\rho(i)}) & \quad
\text{if } \rho = \widetilde{\rho}, \\
(-1)^{s+t} \lambda_{\rho^* \cup \widetilde{\rho}^*} &
\quad \text{if } \rho \cap \widetilde{\rho}^* = \{\rho(s)\}, \rho^* \cap
\widetilde{\rho} = \{\widetilde{\rho}(t)\}, \\
0 & \quad \text{otherwise}.
\end{aligned}
\right.
\end{equation}

Recall that the exterior algebra ${\rm Alt}^k \mathbb{R}^n$ has as a
basis $\mu_{\rho}' := \mu_{\rho(1)}' \wedge \mu_{\rho(2)}' \wedge
\cdots \wedge \mu_{\rho(k)}'$ for $\rho \in \Sigma(k,n)$, where
$\mu_i$ is the orthogonal basis of $\mathbb{R}^n$. Further, the inner
product on ${\rm Alt}^k\mathbb{R}^n$ is defined as (c.f. \cite[pp.~11]{2006ArnoldFalkWinther-a})
\begin{equation} \label{eq:inner-alt}
\langle \omega, \eta \rangle_{{\rm Alt}^k\mathbb{R}^n} = \sum_{\rho \in
\Sigma(k,n)} \omega(\mu_{\rho(1)}, \ldots, \mu_{\rho(k)})
\eta(\mu_{\rho(1)}, \ldots, \mu_{\rho(k)}).
\end{equation}
We have the following result which generalizes~\eqref{eq:atma1} and \eqref{eq:atma2} for arbitrary $n$ and $k$.
\begin{lemma} \label{lm:key-minus}
Let $m_{\rho\widetilde{\rho}}$ be given in \eqref{eq:M-minus}. Then,
\begin{equation} \label{eq:key-minus}
\sum_{\rho, \widetilde{\rho} \in \Sigma(k,n)} a_\rho
a_{\widetilde{\rho}} m_{\rho \widetilde{\rho}} =
\lambda_0 \left( \sum_{\rho \in \Sigma(k,n)}  a_\rho^2
\lambda_{\rho^*} \right)
 + \langle \theta, \theta\rangle_{{\rm Alt}^{n-k+1}\mathbb{R}^n},
\end{equation}
where
$$
\theta = \sum_{\rho \in \Sigma(k,n)} (-1)^{\sigma(\rho)} a_\rho  \sum_{i=1}^k
\sqrt{\lambda_{\rho^*}\lambda_{\rho(i)}} \mu_{\rho^*}' \wedge \mu_{\rho(i)}'.
$$
\end{lemma}
\begin{proof}
We write
$$
\begin{aligned}
& \qquad \langle\theta, \theta\rangle_{{\rm Alt}^{k+1}\mathbb{R}^n}
\\
&  = \Big\langle
\sum_{\rho \in \Sigma(k,n)} (-1)^{\sigma(\rho)} a_\rho \sum_{i=1}^k
\sqrt{\lambda_{\rho^*}\lambda_{\rho(i)}} \mu_{\rho^*}' \wedge
\mu_{\rho(i)}', \\
& \qquad \sum_{\widetilde{\rho} \in \Sigma(k,n)} (-1)^{\sigma(\widetilde\rho)}
a_{\widetilde{\rho}}  \sum_{j = 1}^k
\sqrt{\lambda_{\widetilde{\rho}^*}\lambda_{\widetilde{\rho}(j)} }
\mu_{\widetilde{\rho}^*}' \wedge \mu_{\widetilde{\rho}(j)}' \Big\rangle_{{\rm
Alt}^{n-k+1}\mathbb{R}^n}.
\end{aligned}
$$
Next, we verify \eqref{eq:key-minus} for each component:
\begin{itemize}
\item $\rho = \widetilde{\rho}$, the coefficient of $a_{\rho}^2$ is
$$
\begin{aligned}
& \quad
\Big\langle
\sum_{i=1}^k
\sqrt{\lambda_{\rho^*}\lambda_{\rho(i)}} \mu_{\rho^*}' \wedge \mu_{\rho(i)}',
\sum_{j=1}^k
\sqrt{\lambda_{\rho^*}\lambda_{\rho(j)}} \mu_{\rho^*}' \wedge \mu_{\rho(j)}'
\Big\rangle_{{\rm Alt}^{n-k+1}\mathbb{R}^n} \\
& =  \lambda_{\rho^*} \sum_{i=1}^k
\lambda_{\rho(i)}.
\end{aligned}
$$
\item $|\rho \cap \widetilde{\rho}^*| \geq 2$: the coefficient of $a_{\rho}a_{\widetilde{\rho}}$ is obviously zero.
\item $|\rho \cap \widetilde{\rho}^*| = 1$: Recalling that $\rho \cap \widetilde{\rho}^* =
\{\rho(s)\}, \rho^* \cap \widetilde{\rho} = \{\widetilde{\rho}(t)\}$, we find
the coefficient of $2a_{\rho} a_{\widetilde\rho}$ as
$$
\begin{aligned}
& \quad
(-1)^{\sigma(\rho) + \sigma(\widetilde\rho)}
\Big\langle
\sum_{i=1}^k
\sqrt{\lambda_{\rho^*}\lambda_{\rho(i)}} \mu_{\rho^*}' \wedge \mu_{\rho(i)}',
\sum_{j=1}^k
\sqrt{\lambda_{\widetilde{\rho}^*}\lambda_{\widetilde{\rho}(j)}} \mu_{\widetilde{\rho}^*}' \wedge \mu_{\widetilde{\rho}(j)}'
\Big\rangle_{{\rm Alt}^{n-k+1}\mathbb{R}^n} \\
& = \lambda_{\rho^* \cup \widetilde{\rho}^*}  (-1)^{\sigma(\rho) + \sigma(\widetilde{\rho})}  \langle \mu_{\rho(s)}' \wedge \mu_{\rho^*}',
\mu_{\widetilde{\rho}(t)}' \wedge \mu_{\widetilde{\rho}^*}' \rangle_{{\rm Alt}^{n-k+1}\mathbb{R}^n}
=  (-1)^{s+t} \lambda_{\rho^* \cup \widetilde{\rho}^*}.
\end{aligned}
$$
Here, in the last step, we have used
$$
\begin{aligned}
\rho,\rho^* &= \cdots \rho(s) \cdots, \cdots \widetilde{\rho}(t)\cdots \\
\widetilde{\rho},\widetilde{\rho}^* &= \cdots \widetilde{\rho}(t)\cdots, \cdots \rho(s)\cdots
\end{aligned}
$$
where $\cdots$ represent the common indices, which contribute equally in counting
both $\sigma(\rho)$ and $\sigma(\widetilde{\rho})$.
\end{itemize}
Combining the three cases above and using \eqref{eq:M-minus}, we obtain \eqref{eq:key-minus}.
\end{proof}

The main result in this section is shown next.
\begin{proof}[Proof of Lemma \ref{lm:quasi-minus}]
From \eqref{eq:omega-eta-minus} and \eqref{eq:key-minus}, it is
obvious that the coefficient function of $\operatorname{vol}_T$ for
$\omega \wedge \eta$  does not change sign in $\bar{T}$, as $\lambda_i
\geq 0$. By Lemma \ref{lm:no-change-sign} and equation
\eqref{eq:key-minus}, we have
$$
\lambda_0 \sum_{\rho \in \Sigma(k,n)} a_\rho^2 \lambda_{\rho^*} \equiv 0.
$$
In the interior of $T$, we have $\lambda_i > 0$, $i = 0,\ldots,n$,
which implies that $a_\rho = 0$.
\end{proof}

\section{Unisolvence for the quasi-polynomial spaces of second kind}\label{s:second-kind}
We now prove the unisolvence of the degrees of freedom for quasi-polynomials derived from polynomial
Hilbert complexes of second kind which are defined as
\begin{equation} \label{eq:eP}
e\mathcal{P}_r\Lambda^k(T) := \{
e(\bm{x}) \omega ~|~ \omega \in \mathcal{P}_r\Lambda^k(T)
\}.
\end{equation}

\subsection{Geometrical decomposition of the polynomial spaces of second kind}
We now consider the geometrical decomposition of the polynomial spaces
associated with the Hilbert complexes of second kind: $\mathcal{P}_r\Lambda^k(T)^*$ and
$\mathcal{P}_r\Lambda^k(T)$. We point out that there is a little (if any) analogy in
the proofs for these spaces.

Following~\cite[Section 4.5]{2006ArnoldFalkWinther-a}, the degrees of freedom for
$\mathcal{P}_r\Lambda^k(T)$ are
\begin{equation}\label{eq:dof}
\int_f \Tr_f \omega \wedge \eta, \quad \eta \in \mathcal{P}_{r+k-\dim
f}^-\Lambda^{\dim f - k}(f), \quad f \in \Delta(T).
\end{equation}
Further, the characterization of the trace free part of the
$\mathcal{P}_r\Lambda^k(T)$ is stated below.
\begin{theorem}[Theorem 4.22 in \cite{2006ArnoldFalkWinther-a}]
\label{tm:trace-free}
For $1\leq k \leq n$, $r \geq n+1-k$, the map
\begin{equation}\label{eq:trace-free}
\sum_{\rho \in \Sigma_0(n-k,n)} a_\rho \phi_\rho \mapsto
\sum_{\rho \in \Sigma_0(n-k,n)} a_\rho \lambda_{\rho}
(\rd\lambda)_{\rho^*},
\end{equation}
where the $a_\rho = a_\rho(\lambda_{\rho(0)}, \lambda_{\rho(0)+1},
\ldots, \lambda_n) \in \mathcal{P}_{r+k-n-1}(T)$, defines an
isomorphism of $\mathcal{P}_{r+k-n}^-\Lambda^{n-k}(T)$ onto
$\mathring{\mathcal{P}}_r\Lambda^k(T)$.
\end{theorem}

\subsection{Polynomials of second kind with vanishing trace}
We now state the main result of this section showing
unisolvence of the functionals used for degrees of freedom \eqref{eq:dof}.
\begin{lemma}\label{lm:quasi}
Let $\omega \in \mathring{\mathcal{P}}_r \Lambda^k(T)$. Suppose that
\begin{equation} \label{eq:eP-dofT}
\int e(\bm{x}) \omega \wedge \eta = 0, \qquad \eta \in
\mathcal{P}_{r-n+k}^-\Lambda^{n-k}(T).
\end{equation}
Then $\omega = 0$.
\end{lemma}
Similar to Theorem \ref{tm:quasi-minus}, the induction argument (from
sub-simplices of lower dimension a sub-simplex of higher dimension)
shows that the degrees of freedom in~\eqref{eq:dof} are a
unisolvent set for $e\mathcal{P}_r^-\Lambda^k(T)$, which is stated in
the following theorem.

\begin{theorem} \label{tm:quasi}
Let $0 \leq k \leq n$, $r \geq 1$. Suppose that $\omega \in
\mathcal{P}_r\Lambda^k(T)$ satisfies
$$
\int_f \mathrm{Tr}_f (e(\bm{x}) \omega) \wedge \eta = 0, \quad \eta
\in \mathcal{P}_{r+k-\mathrm{dim}f}^-\Lambda^{\mathrm{dim}f-k}(f),
  \quad f \in \Delta(T).
$$
Then $\omega = 0$.
\end{theorem}

We will give the proof of Lemma \ref{lm:quasi} in the rest of this
section.

\subsection{Calculating the ``mass" matrix for the polynomial spaces of the second kind}
We begin by showing several results that lead to computable form of the mass matrix. The first result is on the Whitney forms depending on general mappings $\rho\in \Sigma_0(n-k,n)$, meaning that $\rho(0)$ is not necessarily $0$.
\begin{lemma} \label{lm:equal}
For any $\rho \in \Sigma_0(n-k,n)$, it holds that
\begin{equation}\label{eq:Whitney3}
\phi_\rho \wedge (\rd\lambda)_{\rho^*} = (-1)^{\sigma(\rho)}
\left( \sum_{i=0}^{n-k} \lambda_{\rho(i)} \right)
\frac{\operatorname{vol}_T}{n!|T|}.
\end{equation}
\end{lemma}
\begin{proof} For the case in which $\rho(0) = 0$, \eqref{eq:Whitney3}
is implied by \eqref{eq:equal-minus}. Next, we consider the case in
which $\rho(0) > 0$, namely $\rho^*(0) = 0$.
\[
\begin{aligned}
\phi_\rho \wedge (\rd\lambda)_{\rho^*} &=
\left( \sum_{i=0}^{n-k} (-1)^i \lambda_{\rho(i)} \rd\lambda_{\rho(0)}
    \wedge \cdots \widecheck{\rd\lambda_{\rho(i)}} \wedge \cdots
    \wedge \rd\lambda_{\rho(n-k)}\right) \wedge (\rd \lambda)_{\rho^*} \\
&= \sum_{i=0}^{n-k} (-1)^{i+1}
\lambda_{\rho(i)} \rd\lambda_{\rho(0)}
    \wedge \cdots \widecheck{\rd\lambda_{\rho(i)}} \wedge \cdots
    \wedge \rd\lambda_{\rho(n-k)} \\
&\qquad\qquad\qquad\quad \wedge \rd\lambda_{\rho(i)} \wedge
\rd\lambda_{\rho^*(1)} \wedge \cdots \wedge\rd\lambda_{\rho^*(k-1)} \\
&= (-1)^{n-k+1} \sum_{i=0}^{n-k}  \lambda_{\rho(i)} (\rd\lambda)_{\rho}
\wedge
\rd\lambda_{\rho^*(1)} \wedge \cdots \wedge\rd\lambda_{\rho^*(k-1)}\\
&= (-1)^{\sigma(\rho)}
\left( \sum_{i=0}^{n-k} \lambda_{\rho(i)} \right)
\frac{\operatorname{vol}_T}{n!|T|}.
\end{aligned}
\]
This completes the proof of the representation in~\eqref{eq:Whitney3}.
\end{proof}

The relation we show next is a key in computing the entries of the mass matrix. Note that Lemma~\ref{lm:cross} below has the same formulation as Lemma~\ref{lm:cross-minus}. However, the different polynomial spaces require different proofs.
\begin{lemma} \label{lm:cross}
For $\rho, \widetilde{\rho} \in \Sigma_0(n-k,n)$ and $\rho \neq \widetilde{\rho}$,
it holds that
\begin{equation}\label{eq:cross}
\begin{aligned}
&~\quad \phi_{\rho} \wedge (\rd\lambda)_{\widetilde{\rho}^*} \\
& =
\left\{
\begin{aligned}
(-1)^{\sigma(\widetilde{\rho}) + s+t} \lambda_{\rho\cap \widetilde{\rho}^*}
\frac{\operatorname{vol}_T}{n!|T|} &
\quad \text{if } \rho \cap \widetilde{\rho}^* = \{\rho(s)\}, \rho^* \cap
\widetilde{\rho} = \{\widetilde{\rho}(t)\}, \\
0 & \quad \text{otherwise}.
\end{aligned}
\right.
\end{aligned}
\end{equation}
\end{lemma}
\begin{proof}
We consider  four possible cases which depend on whether $\rho(0)=0$ or $\widetilde{\rho}(0)=0$.
\paragraph{\bf Case 1:~~} $\rho(0)>0, \widetilde{\rho}(0)=0$. If $|\rho \cap
\widetilde{\rho}^*| \geq 2$, from the definition of Whitney form
\eqref{eq:Whitney}, we easily see that $\phi_{\rho} \wedge
(\rd\lambda)_{\widetilde{\rho}^*} = 0$. The left case is $\rho \cap
\widetilde{\rho}^* = \{\rho(s)\}$ (in this case $\rho^* \cap \widetilde{\rho} = \{0\}
= \{\widetilde{\rho}(0)\}$, namely $t=0$). Then,
$$
\begin{aligned}
\phi_{\rho} \wedge (\rd\lambda)_{\widetilde{\rho}^*} &=
(-1)^s \lambda_{\rho(s)}
\rd\lambda_{\rho(0)}
\wedge \cdots \widecheck{\rd\lambda_{\rho(s)}} \wedge \cdots
\wedge \rd\lambda_{\rho(n-k)} \wedge (\rd\lambda)_{\widetilde{\rho}^*} \\
&= (-1)^s \lambda_{\rho(s)} ({\rd\lambda})_{\widetilde\rho \setminus \{0\}}
\wedge (\rd\lambda)_{\widetilde{\rho}^*} \\
&= (-1)^{\sigma(\widetilde{\rho})+s+0} \lambda_{\rho\cap \widetilde{\rho}^*}
\frac{\operatorname{vol}_T}{n!|T|}.
\end{aligned}
$$
\paragraph{\bf Case 2:~~} $\rho(0)>0, \widetilde{\rho}(0) > 0$. If $|\rho \cap
\widetilde{\rho}^*| \geq 2$, then $\phi_{\rho} \wedge
(\rd\lambda)_{\widetilde{\rho}^*} = 0$ since there only has one
$\rd\lambda_0$ in $(\rd\lambda)_{\widetilde{\rho}^*}$. The left case is
$\rho \cap \widetilde{\rho}^* = \{\rho(s)\}$, $\rho^*\cap \widetilde{\rho} =
\{\widetilde{\rho}(t)\}$. Then,
$$
\begin{aligned}
& ~~\quad \phi_{\rho} \wedge (\rd\lambda)_{\widetilde{\rho}^*} \\
&= (-1)^s\lambda_{\rho(s)}
\rd\lambda_{\rho(0)}
\wedge \cdots \widecheck{\rd\lambda_{\rho(s)}} \wedge \cdots
\wedge \rd\lambda_{\rho(n-k)} \wedge (\rd\lambda)_{\widetilde{\rho}^*} \\
&= (-1)^{s}\lambda_{\rho(s)}
\rd\lambda_{\rho(0)}
\wedge \cdots \widecheck{\rd\lambda_{\rho(s)}} \wedge \cdots
\wedge \rd\lambda_{\rho(n-k)} \\
&\quad\qquad\qquad \wedge (-\rd\lambda_{\widetilde{\rho}(t)}) \wedge
\rd\lambda_{\widetilde{\rho}^*(1)} \wedge \cdots \wedge
\rd\lambda_{\widetilde{\rho}^*(k-1)} \\
&= (-1)^{s+n-k+1} \lambda_{\rho\cap\widetilde{\rho}^*}
\rd\lambda_{\widetilde{\rho}(t)} \wedge
\rd\lambda_{\rho(0)}
\wedge \cdots \widecheck{\rd\lambda_{\rho(s)}} \wedge \cdots
\wedge \rd\lambda_{\rho(n-k)}
(\rd\lambda)_{\widetilde{\rho}^*\setminus\{0\}} \\
&= (-1)^{s+t+n-k+1}\lambda_{\rho\cap\widetilde{\rho}^*}
(\rd\lambda)_{\widetilde{\rho}} \wedge
(\rd\lambda)_{\widetilde{\rho}^*\setminus \{0\}} \\
&= (-1)^{\sigma(\widetilde{\rho})+s+t} \lambda_{\rho\cap\widetilde{\rho}^*}
\frac{\operatorname{vol}_T}{n!|T|}.
\end{aligned}
$$
\paragraph{\bf Case 3:~~} $\rho(0) = 0, \widetilde{\rho}(0)=0$. This case is implied
by Lemma \ref{lm:cross-minus} by considering $\rho \leftarrow \rho
\setminus\{0\}$ and $\widetilde{\rho} \leftarrow \widetilde{\rho}\setminus\{0\}$.
\paragraph{\bf Case 4:~~} $\rho(0) = 0, \widetilde{\rho}(0)>0$. In this case, we have
$\widetilde{\rho}^*(0) = 0$. Note that $\rd\lambda_0 \wedge \rd\lambda_0 =
0$, we have
$$
\phi_{\rho} \wedge (\rd\lambda)_{\widetilde{\rho}^*} = \lambda_0
(\rd\lambda)_{\rho\setminus\{0\}} \wedge
(\rd\lambda)_{\widetilde{\rho}^*}.
$$
If $|\rho \cap \widetilde{\rho}^*| \geq 2$ (or
$|(\rho\setminus\{0\}) \cap \widetilde{\rho}^*| \geq 1$), we have
$\phi_{\rho} \wedge (\rd\lambda)_{\widetilde{\rho}^*} = 0$. Therefore, there is only one case in which
$\phi_{\rho} \wedge (\rd\lambda)_{\widetilde{\rho}^*}$ is nonzero:
$\rho^* \cap \widetilde{\rho} = \{\widetilde{\rho}(t)\}$ and $\rho \cap
    \widetilde{\rho}^* = \{0\} = \{\rho(0)\}$, which gives
\[
\begin{aligned}
\phi_{\rho} \wedge (\rd\lambda)_{\widetilde{\rho}^*} &= \lambda_0
(\rd\lambda)_{\rho\setminus\{0\}} \wedge (-\rd\lambda)_{\widetilde{\rho}(t)}
\wedge \rd\lambda_{\widetilde{\rho}^*(1)}\wedge\cdots \wedge
\rd\lambda_{\widetilde{\rho}^*(k-1)} \\
&= (-1)^{n-k+1}\lambda_0 \rd\lambda_{\widetilde{\rho}(t)}
\wedge (\rd\lambda)_{\rho\setminus\{0\}} \wedge
(\rd\lambda)_{\widetilde{\rho}^*\setminus\{0\}} \\
&= (-1)^{\sigma(\widetilde{\rho})+0+t} \lambda_{\rho \cap \widetilde{\rho}^*}
\frac{\operatorname{vol}_T}{n!|T|}.
\end{aligned}
\]
The relation~\eqref{eq:cross} follows as these case cover all possible choices of $\rho$ and $\widetilde{\rho}$.
\end{proof}

Similarly to the Corollary \ref{cor:symmetry-minus} in the section for the polynomials of first kind, we have the
following result.
\begin{corollary} \label{cor:symmetry}
For any $\rho, \widetilde{\rho} \in \Sigma_0(n-k,n)$, it holds that
$$
(-1)^{\sigma(\rho)} \phi_{\rho} \wedge
\lambda_{\widetilde{\rho}} (\rd\lambda)_{\widetilde{\rho}^*} =
(-1)^{\sigma(\widetilde \rho)} \phi_{\widetilde{\rho}} \wedge
\lambda_{\rho} (\rd\lambda)_{\rho^*}.
$$
\end{corollary}
\begin{proof}
We only prove the case in which $\rho\cap \widetilde{\rho}^* = \{\rho(s)\}$
and $\rho^* \cap \widetilde{\rho} = \{\widetilde{\rho}(t)\}$, as the other cases
are trivial. By Lemma \ref{lm:cross},
$$
\begin{aligned}
\text{LHS} & = (-1)^{\sigma(\rho)}\lambda_{\widetilde{\rho}}
(-1)^{\sigma(\widetilde{\rho})+s+t} \lambda_{\rho \cap \widetilde{\rho}^*}
\frac{\operatorname{vol}_T}{n!|T|} \\
& =
(-1)^{\sigma(\rho)+\sigma(\widetilde{\rho})+s+t} \lambda_{\rho\cup
  \widetilde{\rho}}
\frac{\operatorname{vol}_T}{n!|T|} = \text{RHS}.
\end{aligned}
$$
This completes the proof.
\end{proof}

\subsection{Proof of Lemma \ref{lm:quasi}} We now consider a polynomial $k$-form $\omega \in
\mathring{\mathcal P}_r \Lambda^k(T)$. Theorem \ref{tm:trace-free} implies that this form can be
uniquely represented as
$$
\omega = \sum_{\rho\in \Sigma_0(n-k,n)} a_\rho \lambda_\rho
(\rd\lambda)_{\rho^*},
$$
where $a_\rho = a_{\rho}(\lambda_{\rho(0)}, \lambda_{\rho(0)+1},
\ldots, \lambda_n) \in \mathcal{P}_{r+k-n-1}(T)$. As we have done
earlier, we pick a special test form $\eta$ in \eqref{eq:eP-dofT}
defined by
$$
\eta = (-1)^{k(n-k)}\sum_{\rho \in \Sigma_0(n-k,n)} (-1)^{\sigma(\rho)} a_\rho \phi_\rho.
$$
Then,  Lemma \ref{lm:equal} and Lemma \ref{lm:cross}, imply that
\begin{equation} \label{eq:omega-eta}
\omega \wedge \eta = \sum_{\rho, \widetilde{\rho}\in \Sigma_0(n-k,n)}
a_{\rho} a_{\widetilde{\rho}} m_{\rho\widetilde{\rho}}
\frac{\operatorname{vol}_T}{n!|T|}.
\end{equation}
where $M = (m_{\rho\widetilde{\rho}})$. Now, Corollary~\ref{cor:symmetry} implies that $M$ is symmetric and, moreover, we have that
\begin{equation} \label{eq:M}
m_{\rho\widetilde{\rho}} =
\left\{
\begin{aligned}
\lambda_{\rho} (\sum_{i=0}^{n-k}\lambda_{\rho(i)}) & \quad
\text{if } \rho = \widetilde{\rho}, \\
(-1)^{\sigma(\rho) + \sigma(\widetilde{\rho}) + s+t} \lambda_{\rho \cup
\widetilde{\rho}} &
\quad \text{if } \rho \cap \widetilde{\rho}^* = \{\rho(s)\}, \rho^* \cap
\widetilde{\rho} = \{\widetilde{\rho}(t)\}, \\
0 & \quad \text{otherwise}.
\end{aligned}
\right.
\end{equation}

To prove the next result we need the canonical Euclidean basis
$\mathbb{R}^{n+1}$ which we denote by $\mu_0, \ldots, \mu_n$.
\begin{lemma} \label{lm:key}
Let $m_{\rho\widetilde{\rho}}$ be given in \eqref{eq:M}. It holds that
\begin{equation} \label{eq:key}
\sum_{\rho, \widetilde{\rho}\in \Sigma_0(n-k,n)} a_{\rho} a_{\widetilde{\rho}}
m_{\rho\widetilde{\rho}} =
\langle\theta, \theta\rangle_{\mathrm{Alt}^{k+1}\mathbb{R}^{n+1}},
\end{equation}
where
$$
\theta = \sum_{\rho\in\Sigma_0(n-k,n)} a_\rho \sum_{i=0}^{n-k}
\sqrt{\lambda_\rho \lambda_{\rho(i)}} \mu_{\rho(i)}' \wedge
\mu_{\rho^*}'.
$$
\end{lemma}
\begin{proof}
It is straightforward to show that the coefficient of $a_{\rho}^2$ of
the right hand side of \eqref{eq:key} is $m_{\rho\rho}$. If $\rho \cap
\widetilde{\rho}^* = \{\rho(s)\}$ and $\rho^*\cap\widetilde{\rho} =
\{\widetilde{\rho}(t)\}$, then
$$
\begin{aligned}
& \Big\langle a_\rho \sum_{i=0}^{n-k}
\sqrt{\lambda_\rho \lambda_{\rho(i)}} \mu_{\rho(i)}' \wedge
\mu_{\rho^*}',
a_{\widetilde{\rho}}
\sum_{j=0}^{n-k}
\sqrt{\lambda_{\widetilde\rho} \lambda_{\widetilde{\rho}(j)}}
\mu_{\widetilde{\rho}(j)}' \wedge \mu_{\widetilde{\rho}^*}' \Big\rangle \\
= ~&a_{\rho}a_{\widetilde{\rho}} \Big\langle
\sqrt{\lambda_\rho \lambda_{\rho(s)}} \mu_{\rho(s)}' \wedge
\mu_{\rho^*}',
\sqrt{\lambda_{\widetilde\rho} \lambda_{\widetilde{\rho}(t)}}
\mu_{\widetilde{\rho}(t)}' \wedge \mu_{\widetilde{\rho}^*}'
\Big\rangle \\
= ~&a_{\rho}a_{\widetilde{\rho}} \lambda_{\rho \cup \widetilde{\rho}} \Big\langle
\mu_{\rho(s)}' \wedge \mu_{\rho^*}',
\mu_{\widetilde{\rho}(t)}' \wedge \mu_{\widetilde{\rho}^*}'
\Big\rangle \\
= ~& (-1)^{\sigma(\rho)+\sigma(\widetilde{\rho})+s+t} \lambda_{\rho \cup
  \widetilde{\rho}^*}.
\end{aligned}
$$
Both left and right hand sides of \eqref{eq:key} obviously vanish when
$|\rho \cap \widetilde{\rho}^*| \geq 2$.
\end{proof}
We are ready to show the main result of this section.
\begin{proof}[Proof of Lemma \ref{lm:quasi}]
From \eqref{eq:omega-eta} and \eqref{eq:key}, it is obvious that the
coefficients in front of $\operatorname{vol}_T$ in the product $\omega \wedge \eta$
do not change sign in $\bar{T}$, as $\lambda_i|_{\bar{T}} \geq 0$.
By Lemma \ref{lm:no-change-sign}, we have that $\theta \equiv 0$. This shows that
\begin{equation} \label{eq:vanish-theta}
\sum_{\rho \in \Sigma_0(n-k,n)} a_\rho \sum_{i=0}^{n-k}
\sqrt{\lambda_\rho \lambda_{\rho(i)}} \mu_{\rho(i)}'\wedge \mu_{\rho^*}'
\equiv 0,
\end{equation}
where $a_\rho = a_\rho(\lambda_{\rho(0)}, \lambda_{\rho(0)+1}, \ldots,
\lambda_{n}) \in \mathcal{P}_{r+k-n-1}(T)$. Next, we show that
$a_\rho \equiv 0$ by \eqref{eq:vanish-theta} and an {\it induction
argument} on $\rho(0)$ (the minimal index of $\rho$).

\paragraph{\bf Step 1.~~} We first assume that
$\rho(0) = k$, namely $\mathcal{R}(\rho) = \{k, k+1,\ldots n\}$. Collecting the coefficients in front of $\mu_{\rho^*\cup
\{k\}}'$ (namely $\mu_0'\wedge\cdots\wedge \mu_k'$) in \eqref{eq:vanish-theta}, we have
$$
(-1)^k a_\rho \sqrt{\lambda_\rho \lambda_k} + \sum_{i=0}^{k-1}(-1)^i a_{\rho\cup\{i\}\setminus\{k\}}
\sqrt{\lambda_{\rho\cup\{i\}\setminus\{k\}} \lambda_i} \equiv 0.
$$
Restricting the above identity on the simplex $f_{\rho}$,
and recalling that $\lambda_i|_{f_\rho} = 0, 0 \leq i \leq (k-1)$, then it shows that
$a_\rho|_{f_\rho} \equiv 0$. Notice that
\[
a_\rho = a_\rho(\lambda_{\rho(0)}, \lambda_{\rho(0)+1}, \ldots, \lambda_n) = a_\rho(\lambda_k, \ldots, \lambda_n).
\]
We conclude that $a_\rho = 0$ because the extension $E_{f_{\rho}, T}$ is injective.

\paragraph{\bf Step 2.~~} Assume that $a_\rho = 0$ for $\ell < \rho(0) \leq k$. Then, for any $\rho\in \Sigma_0(n-k,n)$ with $\rho(0) = \ell$,
we collect the coefficients of $\mu_{\rho^* \cup \{\ell\}}'$ in \eqref{eq:vanish-theta}, which gives
$$
\pm a_\rho \sqrt{\lambda_\rho \lambda_\ell} +  \sum_{\widetilde{\rho}(0) < \ell} \pm a_{\widetilde{\rho}}\sqrt{\lambda_{\widetilde{\rho}}\lambda_{\widetilde{\rho} \cap \rho^*}}
+ \sum_{\widetilde{\rho}(0) > \ell} \pm a_{\widetilde{\rho}}\sqrt{\lambda_{\widetilde{\rho}}\lambda_{\widetilde{\rho} \cap \rho^*}} = 0.
$$
Restricting the above equality on $f := [x_\ell, x_{\ell+1}, \ldots, x_n]$. Notice that for $\widetilde{\rho}(0) < \ell$, we have $\lambda_{\widetilde{\rho}}|_f = 0$; while for $\widetilde{\rho}(0) > \ell$, we have
$a_{\widetilde{\rho}} = 0$ by the inductive assumption. Hence, $a_\rho|_f = 0$, which gives $a_\rho = 0$ due to the fact that $a_\rho = a_\rho(\lambda_\ell, \lambda_{\ell+1}, \ldots, \lambda_n)$.

Combining the results from {\bf Step 1} and {\bf Step 2} complete the proof of the lemma.
\end{proof}

\section{An application: exponential fitting for general
convection-diffusion problems}\label{s:applications}

In this section we give a derivation of the simplex-averaged finite
element (SAFE) scheme of arbitrary order for the general
convection-diffusion problems.

\subsection{General convection-diffusion problems}
Let ${\beta}({\bm x})$ be a given vector field and consider the
general convection-diffusion problem in the following form:
\begin{equation} \label{equ:conv-diff}
\left\{
\begin{array}{ll}
\mathcal{L}u := \rd^*(\alpha \rd u + i^*_{{\beta}} u) + \gamma u = f
&\quad \text{in } \Omega,\\
\Tr u = 0 & \quad \text{on }\Gamma_0 \subset \partial \Omega, \\
\Tr[\star (\alpha \rd u + i_{{\beta}}^*u)] = g & \quad \text{on }
\Gamma_{N} = \partial \Omega \setminus \Gamma_0. \\
\end{array}
\right.
\end{equation}
We assume that $\alpha$, ${\beta}$ and $\gamma$ are piecewise smooth
functions on $\bar{\Omega}$ and $\alpha(x) \geq \alpha_0 > 0$,
$\gamma(x) \geq 0$.  Here, $\rd$, $\rd^*$, $i_{{\beta}}$,
$i_{{\beta}}^*$, $\star$, $\Tr$ denote the coderivative,
contraction, dual of contraction, Hodge star, and trace operator,
respectively (cf.  \cite{2006ArnoldFalkWinther-a}).
\begin{table}[!htbp]
\centering
\begin{tabular}{m{0.6cm}m{3.0cm}m{3.4cm}m{1.1cm}m{0.8cm}m{0.8cm}}
\hline
$k$ & $\rd u$ & $\rd^*u$ & $i_{\beta}u$ & $i_{\beta}^* u$ & $\mathrm{Tr}$\\
\hline\hline
$0$ & $\grad u$ (or $\nabla u$) & $-\mydiv u$ (or $-\nabla \cdot u$) & & $\beta u$ & $u$ \\
$1$ & $\curl u$ (or $\nabla \times u$) & $\curl u$ (or $\nabla \times
    u$) & $\beta\cdot u$ & $\beta\times u$ & $\nu \times u$ \\
$2$ & $\mydiv u$ (or $\nabla \cdot u$) & $-\grad u$ (or $-\nabla u$) &
$-\beta\times u$ & $\beta\cdot u$ & $u\cdot \nu$ \\
$3$ & & & $\beta u$ &
&  \\ \hline
\end{tabular}
\caption{Translation table in 3D.}
\label{tab:proxy-3D}
\end{table}

We note the identification between differential forms and vector
proxies in $\mathbb{R}^3$, outlined in Table \ref{tab:proxy-3D}.  The
specific examples  corresponding to Table \ref{tab:proxy-3D} are
listed as follows.

\begin{enumerate}
\item For $k=0$, we have the $H(\grad)$ convection-diffusion problem:
\begin{equation} \label{equ:grad-conservative}
\left\{
\begin{array}{ll}
-\nabla \cdot (\alpha \nabla u + \beta u) + \gamma u = f &\quad \text{in }
\Omega,\\
u = 0 & \quad \text{on }\Gamma_0 \subset \partial \Omega, \\
(\alpha \nabla u + \beta u)\cdot \nu = g & \quad \text{on }
\Gamma_{N} = \partial \Omega \setminus \Gamma_0. \\
\end{array}
\right.
\end{equation}

\item For $k=1$, we have the $H(\curl)$ convection-diffusion problem:
\begin{equation} \label{equ:curl-conservative}
\left\{
\begin{array}{ll}
\nabla \times(\alpha \nabla \times u + \beta \times u) + \gamma u = f
&\quad \text{in }
\Omega,\\
\nu \times u = 0 & \quad \text{on }\Gamma_0 \subset \partial \Omega, \\
\nu \times (\alpha \nabla \times u + \beta \times u) = g & \quad \text{on }
\Gamma_{N} = \partial \Omega \setminus \Gamma_0. \\
\end{array}
\right.
\end{equation}

\item For $k=2$, we have the $H(\mydiv)$ convection-diffusion problem:
\begin{equation} \label{equ:div-conservative}
\left\{
\begin{array}{ll}
-\nabla (\alpha \nabla \cdot u + \beta \cdot u) + \gamma u = f
&\quad \text{in }
\Omega,\\
u \cdot \nu = 0 & \quad \text{on }\Gamma_0 \subset \partial \Omega, \\
\alpha \nabla \cdot u + \beta \cdot u = g & \quad \text{on }
\Gamma_{N} = \partial \Omega \setminus \Gamma_0. \\
\end{array}
\right.
\end{equation}
We note that the $H({\rm curl})$ or $H({\rm div})$
convection-diffusion problem usually arises from the
magnetohydrodynamics (cf. \cite{gerbeau2006mathematical}).
\end{enumerate}

We introduce the space of vanishing trace on $\Gamma_0$ as
$$
V := \{w \in H\Lambda^k(\Omega): ~ \Tr w= 0~\text{on } \Gamma_0\},
$$
equipped with the norm $\|w\|_{H\Lambda,\Omega}^2 :=
\|w\|_{0,\Omega}^2 + \|\rd w\|_{0,\Omega}^2$.  Then, the variational
formulation for \eqref{equ:conv-diff} is: Find $u \in V$ such that
\begin{equation} \label{equ:variational}
a(u, v) = F(v) \qquad \forall v\in V,
\end{equation}
where
$$
a(u, v) := (\alpha \rd u + i_{\beta}^* u, \rd v) + (\gamma u, v), \quad
F(v) := (f, v) + \langle g, \Tr v \rangle_{\Gamma_N}.
$$

\subsection{Exponential fitting for the flux}
Let $\theta(\bm{x}) = \beta(\bm{x}) / \alpha(\bm{x})$, the
bilinear form $a(\cdot,\cdot)$ can be rewritten as
\begin{equation} \label{eq:bilinear-J}
a(u, v) = (\alpha J_\theta u, \rd v) + (\gamma u, v),
\end{equation}
where the flux $J_\theta u := \rd u + i_\theta^* u$.  We generalize
the identity of $J_\theta$ in \cite[Lemma 3.1]{2019WuXu-a} as follows.
Given a scalar function $\psi(\bm{x})$, we have, by the Leibniz rule,
that
$$
  \rd (\mathrm{e}^\psi u) = \rd \mathrm{e}^{\psi} \wedge u +
    \mathrm{e}^\psi \wedge \rd u = \mathrm{e}^\psi ( \nabla \psi
        \wedge u + \rd u) = \mathrm{e}^\psi J_{\nabla \psi} u,
$$
which gives
\begin{equation} \label{eq:exp-fitting}
J_{\nabla \psi} u = \mathrm{e}^{-\psi} \rd (\mathrm{e}^\psi u).
\end{equation}
In other words, the flux $J_{\nabla \psi}$ can be represented through
the following diagram.

\begin{center}
\begin{tikzpicture}[node distance=9em]
  \node (Form) {$C^\infty \Lambda^{k}(\Omega)$};
  \node (PlusForm) [right of= Form] {$C^\infty\Lambda^{k+1}(\Omega)$};
  \draw[->] (Form) to node [above] {$\rd$} (PlusForm);

  \node (JForm) [below of=Form] {$C^\infty\Lambda^k(\Omega)$};
  \node (JPlusForm) [right of=JForm] {$C^\infty\Lambda^{k+1}(\Omega)$};
  \draw[->] (JForm) to node [above] {$J_{\nabla \psi}$} (JPlusForm);

  \draw[->] (Form.255) to node [left]{$\mathrm{e}^{-\psi}$} (JForm.105);
  \draw[->] (JForm.75) to node [right]{$\mathrm{e}^{\psi}$} (Form.285);
  \draw[->] (PlusForm.255) to node [left]{$\mathrm{e}^{-\psi}$}
  (JPlusForm.105);
  \draw[->] (JPlusForm.75) to node [right]{$\mathrm{e}^{\psi}$}
  (PlusForm.285);
\end{tikzpicture}
\end{center}

\subsection{Defining numerical flux using quasi-polynomial spaces}
The way to define the numerical flux mimics the above diagram
at discrete level. In the first step, we use the polynomial
differential form to approximate the continuous one. We also denote
the local canonical interpolation as $\Pi_{T}^k: C^\infty
\Lambda^k(\Omega) \to \polyspace{}{}\Lambda^k(T)$.
\begin{center}
\begin{tikzpicture}[node distance=9em]
  \node (Form) {$\polyspace{}{}\Lambda^{k}(T)$};
  \node (PlusForm) [right of= Form] {$\polyspace{}{}\Lambda^{k+1}(T)$};
  \draw[->] (Form) to node [above] {$\rd$} (PlusForm);

  \node (JForm) [below of=Form]  {$\polyspace{}{}\Lambda^{k}(T)$};
  \node (JPlusForm) [right of=JForm] {$\polyspace{}{}\Lambda^{k+1}(T)$};

  \draw[->] (JForm.75) to node [right]{$\Pi_{T}^k \mathrm{e}^{\psi}$} (Form.285);
  \draw[->] (JPlusForm.75) to node [right]{$\Pi_{T}^{k+1}\mathrm{e}^{\psi}$}
  (PlusForm.285);
\end{tikzpicture}
\end{center}

In the second step, we need to show that the operator $\Pi_{T}^k \mathrm{e}^\psi:
\polyspace{}{}\Lambda^k(T) \to \polyspace{}{}\Lambda^k(T)$ is an
isomorphism. It suffices to check that $\Pi_{T}^k \mathrm{e}^\psi$
is an injection, which is readily shown, respectively, for $\mathcal{P} = \mathcal{P}_r^-$ in Section~\ref{s:first-kind}
and $\mathcal{P} = \mathcal{P}_r$ in Section~\ref{s:second-kind}.
Hence, the inverse of $\Pi_{T}^k \mathrm{e}^\psi$ on $\mathcal{P}\Lambda^k(T)$ (denoted by
$H_{T}^k$) exists. Then, the numerical flux is defined by
\begin{equation} \label{eq:numerical-flux}
J_{\nabla \psi, T} v_h := H_{T}^{k+1} \rd \Pi_{T}^k
\mathrm{e}^{\psi} v_h.
\end{equation}

\begin{center}
\begin{tikzpicture}[node distance=10em]
  \node (Form) {$\polyspace{}{}\Lambda^{k}(T)$};
  \node (PlusForm) [right of= Form] {$\polyspace{}{}\Lambda^{k+1}(T)$};
  \draw[->] (Form) to node [above] {$\rd$} (PlusForm);

  \node (JForm) [below of=Form]  {$\polyspace{}{}\Lambda^{k}(T)$};
  \node (JPlusForm) [right of=JForm] {$\polyspace{}{}\Lambda^{k+1}(T)$};

  \draw[->] (Form.255) to node [left]{$H_{T}^k$} (JForm.105);
  \draw[->] (JForm.75) to node [right]{$\Pi_{T}^k \mathrm{e}^{\psi}$} (Form.285);
  \draw[->] (PlusForm.255) to node [left]{$H_{T}^{k+1}$}
  (JPlusForm.105);
  \draw[->] (JPlusForm.75) to node [right]{$\Pi_{T}^{k+1}\mathrm{e}^{\psi}$}
  (PlusForm.285);

  \draw[->] (JForm) to node [above] {$J_{\nabla \psi, T}$} (JPlusForm);
\end{tikzpicture}
\end{center}

For the general convection $\theta(\bm{x})$, we take a {\it piecewise
constant approximation} $\bar{\theta}$, namely, $\bar{\theta}|_T$ is a
constant vector for every $T \in \mathcal{T}_h$. Taking $\psi(\bm{x})
= \bar{\theta}(\bm{x}) \cdot \bm{x}$ in \eqref{eq:numerical-flux}, the
bilinear form \eqref{eq:bilinear-J} has the approximation
\begin{equation} \label{eq:SAFE-a}
\begin{aligned}
a_h(u_h, v_h) & =
\sum_{T\in \mathcal{T}_h} (\alpha J_{\bar{\theta},T}
    u_h, \rd v_h)_T + (i_{\beta - \alpha\bar{\theta}}^* u_h, \rd v_h)_T \\
&\quad + (\gamma u_h, v_h) + s_h(u_h, v_h).
\end{aligned}
\end{equation}
where $s_h(\cdot, \cdot)$ is a proper stabilization term. In this
paper, we simply take $s_h(\cdot,\cdot) = 0$, which is acceptable for
many cases. It is our future work to design and analysis $s_h(\cdot,
\cdot)$ for general $\alpha$ and $\beta$. When the diffusion
coefficient $\alpha$ is piecewise constant, we have $\bar{\theta} =
\bar{\beta} / \alpha$ and therefore $\beta - \alpha \bar{\theta} =
\beta - \bar{\beta}$, which corresponds to convection speed of the
local perturbation of $\beta$. In particular, with the piecewise
constant diffusion coefficient $\alpha$ and convection speed $\beta$,
the term $i_{\beta - \alpha\bar{\theta}}^* u_h$ vanishes. We emphasis
that the first part in \eqref{eq:SAFE-a}, namely $(\alpha
J_{\bar{\theta}, T} u_h, \rd v_h)_T$, can be discretized via
simplex-averaged finite element (SAFE) method \cite{2019WuXu-a}. The
construction above gives a promising way for deriving higher order
SAFE schemes which can approximate accurately in the convection
dominating case for $k$-forms.

\section{Numerical tests}\label{s:numerics}


In this section, we test the performance of the exponential fitting scheme using polynomials of  degree $\le 2$ for the scalar convection-diffusion
equation~\eqref{equ:grad-conservative}. The following discrete de Rham
sequence is applied:
\begin{equation} \label{eq:P2-sequence}
\mathcal{P}_2 \Lambda^0 \xrightarrow{{\rm grad}} \mathcal{P}_1
\Lambda^1 \xrightarrow{{\rm curl}} \mathcal{P}_0\Lambda^2,
\end{equation}
where the 2D ${\rm curl}$ operator is defined by $\curl v =
(\partial_{y} v, -\partial_{x} v)^T$.

Below we report two sets of numerical tests: one on the the convergence order of the exponential fitting scheme and the other on the performance of this scheme in the
convection dominating case. In all tests, we take $\gamma = 0$ in \eqref{equ:grad-conservative}, i.e.,
\begin{equation} \label{eq:conv-diff-test}
-\nabla \cdot(\alpha \nabla u + \beta u ) = f.
\end{equation}
The computational domain is the square $\Omega = (0,1)^2$, and we impose Dirichlet boundary conditions on the boundary $\partial\Omega$.
We use uniform meshes with varying mesh sizes for all numerical tests and consider piecewise constant diffusion coefficients. In such case we have
$\bar{\theta}|_T = (\beta(x_c) / \alpha)|_T$
on each element $T$, where $x_c$ is the barycenter of $T$.

Finally, the implementation issues pertinent to the computation of the numerical fluxes $(\alpha J_{\bar{\theta},T}u_h, \nabla
v_h)_T$ in \eqref{eq:SAFE-a}, which hinge on computing the generalized Bernoulli functions are discussed in Appendix~\ref{sec:implementation}. By computing the corresponding limits we show that the resulting scheme is a special upwind scheme for the limiting case of vanishing diffusion coefficient.

\subsection{Convergence order test} In the first
set of examples we consider a scalar convection diffusion equation with exact solution
$$
u = \mathrm{e}^{x-y}\sin(\pi x) \cos(\pi y).
$$
The constant diffusion
coefficients range from $10$ to $10^{-5}$ and
the convection speed is set to $\beta = (1, 2)$ or $\beta = (-y,x)$. The right hand side $f$ for each example is computed  using this data.
\begin{table}[!htbp]
\centering
\captionsetup{justification=centering}
\subfloat[$\alpha=10$]{
\footnotesize
\begin{tabular}{|l|cc|cc|}
\hline
$1/h$ & $\|\epsilon_h\|_0$ & $h^n$ &
$|\epsilon_h|_1$ & $h^n$\\ \hline
4  & 7.696e-03  & --- & 1.143e-01  & --- \\
8  & 9.676e-04  & 2.99 & 2.914e-02 & 1.97 \\
16 & 1.218e-04  & 2.99 & 7.320e-03 & 1.99 \\
32 & 1.531e-05 & 2.99 & 1.832e-03  & 2.00 \\
64 & 1.945e-06  & 2.98  & 4.582e-04 & 2.00 \\  \hline
\end{tabular}
\label{tab:constant-beta1}
}
\subfloat[$\alpha=10^{-1}$] {
\footnotesize
\begin{tabular}{|l|cc|cc|}
\hline
$1/h$ & $\|\epsilon_h\|_0$ & $h^n$ &
$|\epsilon_h|_1$ & $h^n$\\ \hline
 4   & 2.990e-02  & ---  & 6.537e-01   &--- \\
 8   & 3.445e-03   &3.12   &1.710e-01   &1.93 \\
 16 & 2.920e-04   &3.56   &2.935e-02   &2.54 \\
 32 & 4.432e-05   &2.72   &4.309e-03   &2.77 \\
64 & 1.118e-05   & 1.99   &6.839e-04   & 2.66 \\ \hline
\end{tabular}
\label{tab:constant-beta2}
} \\
\subfloat[$\alpha=10^{-3}$] {
\footnotesize
\begin{tabular}{|l|cc|cc|}
\hline
$1/h$ & $\|\epsilon_h\|_0$ & $h^n$ &
$|\epsilon_h|_1$ & $h^n$\\ \hline
4   & 5.733e-02   & ---  & 1.075e+00  & --- \\
8   & 1.435e-02   & 2.00  & 5.354e-01  & 1.01 \\
16 & 3.449e-03   & 2.06  & 2.645e-01  & 1.02 \\
32 & 8.206e-04   & 2.07  & 1.300e-01  & 1.02 \\
64 & 1.910e-04   & 2.10  & 6.297e-02  & 1.05 \\ \hline
\end{tabular}
\label{tab:constant-beta3}
}
\subfloat[$\alpha=10^{-5}$] {
\footnotesize
\begin{tabular}{|l|cc|cc|}
\hline
$1/h$ & $\|\epsilon_h\|_0$ & $h^n$ &
$|\epsilon_h|_1$ & $h^n$\\ \hline
4   & 5.769e-02 &  --- & 1.079e+00 & --- \\
8   & 1.457e-02  & 1.99 & 5.390e-01 & 1.00 \\
16  & 3.560e-03 & 2.03 & 2.679e-01 & 1.01 \\
32  & 8.745e-04 & 2.03 & 1.335e-01 & 1.00 \\
64  & 2.162e-04 & 2.02 & 6.664e-02 & 1.00 \\ \hline
\end{tabular}
\label{tab:constant-beta4}
}
\caption{The error, $\epsilon_h = u - u_h$, and convergence order for
$\beta = (1, 2)$.} \label{tab:constant-beta}
\end{table}

For the constant convection speed, we observe from Table
\ref{tab:constant-beta} that, $\|u - u_h\|_{H^1} = \mathcal{O}(h^2)$
and  $\|u - u_h\|_{L^2} = \mathcal{O}(h^3)$ for diffusion dominating
case ($\alpha = 10$). For the convection dominated case ($\alpha = 10^{-5}$), we
observe a first-order convergence in $H^1$ norm and second-order
convergence in $L^2$ norm, which is sub-optimal. The convergence
orders change accordingly with the transition from diffusion dominating
case to convection dominating case. Moreover, for a solution without
boundary or internal layer, the convergence orders can be clearly
observed for coarse meshes regardless of the magnitude of diffusion
coefficient. A similar result can be observed for the variable
convection speed; see Table \ref{tab:variable-beta}.

\begin{table}[!htbp]
\centering
\captionsetup{justification=centering}
\subfloat[$\alpha=10$]{
\footnotesize
\begin{tabular}{|l|cc|cc|}
\hline
$1/h$ & $\|\epsilon_h\|_0$ & $h^n$ &
$|\epsilon_h|_1$ & $h^n$\\ \hline
4   & 7.691e-03  & ---  & 1.142e-01 &  --- \\
8   & 9.671e-04  & 2.99  & 2.913e-02  & 1.97 \\
16 & 1.216e-04  & 2.99  & 7.320e-03  & 1.99 \\
32 & 1.525e-05  & 3.00  & 1.832e-03  & 2.00 \\
64 & 1.914e-06  & 2.99  & 4.582e-04  & 2.00 \\ \hline
\end{tabular}
\label{tab:variable-beta1}
}
\subfloat[$\alpha=10^{-1}$] {
\footnotesize
\begin{tabular}{|l|cc|cc|}
\hline
$1/h$ & $\|\epsilon_h\|_0$ & $h^n$ &
$|\epsilon_h|_1$ & $h^n$\\ \hline
  4   &  8.037e-03 &  ---  &  1.405e-01 &  ---       \\
8   &  1.171e-03 &  2.78 &  3.269e-02 &  2.10 \\
16  &  2.590e-04 &  2.18 &  7.798e-03 &  2.07 \\
32  &  6.358e-05 &  2.03 &  1.917e-03 &  2.02 \\
64  &  1.584e-05 &  2.00 &  4.772e-04 &  2.01 \\ \hline
\end{tabular}
\label{tab:variable-beta2}
} \\
\subfloat[$\alpha=10^{-3}$]{
\footnotesize
\begin{tabular}{|l|cc|cc|}
\hline
$1/h$ & $\|\epsilon_h\|_0$ & $h^n$ &
$|\epsilon_h|_1$ & $h^n$\\ \hline
4   &  4.731e-02  & ---    &  9.355e-01 & ---   \\
8   & 1.142e-02   & 2.05 & 5.796e-01 & 0.69 \\
16  & 2.372e-03  & 2.27 & 2.743e-01 & 1.08 \\
32  & 3.713e-04  & 2.68 & 1.012e-01 & 1.44 \\
64   & 6.023e-05 & 2.62 & 3.556e-02 & 1.51 \\ \hline
\end{tabular}
\label{tab:variable-beta3}
}
\subfloat[$\alpha=10^{-5}$] {
\footnotesize
\begin{tabular}{|l|cc|cc|}
\hline
$1/h$ & $\|\epsilon_h\|_0$ & $h^n$ &
$|\epsilon_h|_1$ & $h^n$\\ \hline
4    & 4.880e-02 &  ---     & 9.557e-01 &  ---     \\
8    & 1.197e-02 &  2.03  & 5.917e-01 &  0.69 \\
16  & 2.757e-03 &  2.12  & 3.162e-01 &  0.90 \\
32  & 6.593e-04 &  2.06  & 1.611e-01 &  0.97  \\
64  & 1.661e-04 &  1.99  & 8.215e-02 &  0.97 \\ \hline
\end{tabular}
\label{tab:variable-beta4}
}
\caption{The error, $\epsilon_h = u - u_h$, and convergence order for
$\beta = (-y, x)$. } \label{tab:variable-beta}
\end{table}

\subsection{Solutions with interior or boundary layers}
We next test the performance of the exponential fitting scheme for problems whose solutions exhibit interior or boundary layers. We consider the equation
\eqref{eq:conv-diff-test} subject to the homogeneous Dirichlet
boundary conditions. Again,
we take $\beta = (1, 2)$, $f = 1$, and fix the mesh size $h = 2^{-6}$. In this set of examples, we vary the diffusion coefficient as
$$
\text{Case 1: } \alpha = 10^{-6}, \qquad
\text{Case 2: }
\alpha =
\begin{cases}
1 & x < 0.5, \\
  10^{-3} & x > 0.5.
\end{cases}
$$
The numerical solutions are shown in  Figures \ref{fig:bdlayer} and \ref{fig:inlayer}.  For the constant diffusion coefficient case, the
ratio $h/\alpha = 15625$, which is rather large compared to the convection speed.  It is clearly seen that there are no spurious  oscillations or smearing near the boundary layer (Case 1) or the internal layer
(Case 2). We remark that for high order exponential fitting schemes, which are not necessarily monotone, stabilization terms may be needed in the bilinear form, especially in case of general $\alpha$ and $\beta$.

\begin{figure}[!htbp]
\centering
\captionsetup{justification=centering}
\subfloat[Case 1: boundary layer]{
  \includegraphics[width=0.49\textwidth]{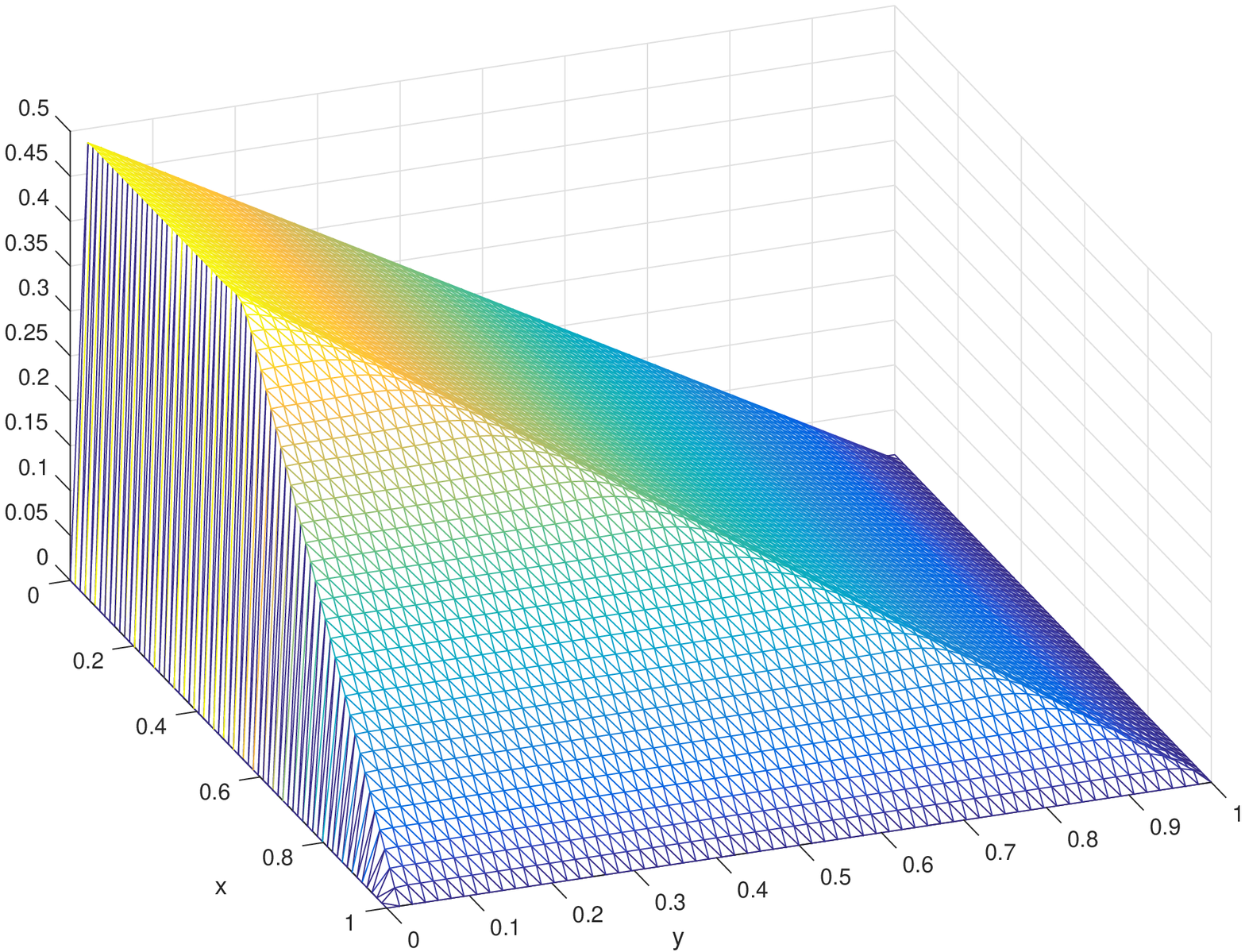}
  \label{fig:bdlayer}
}%
\subfloat[Case 2: interior layer]{
  \includegraphics[width=0.49\textwidth]{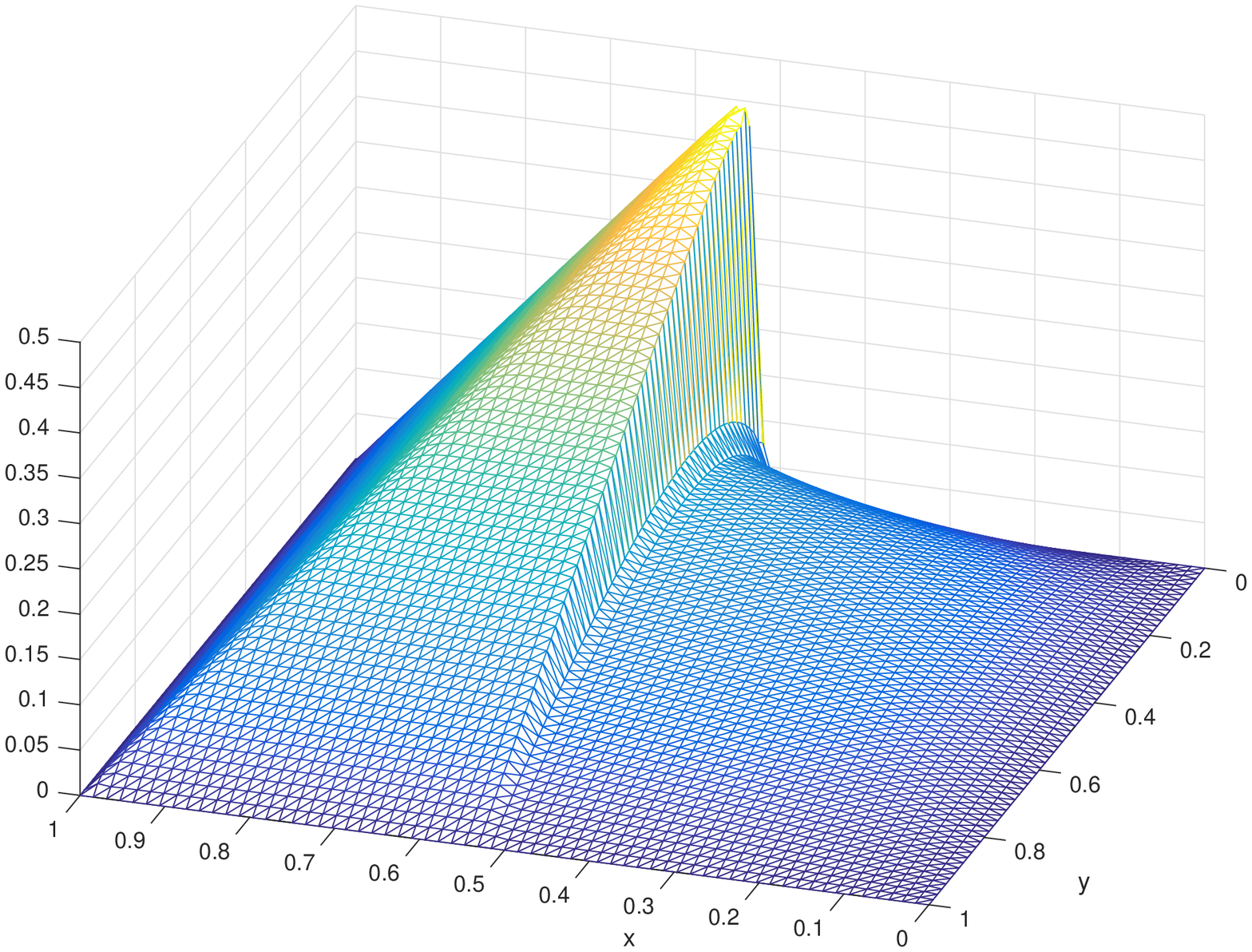}
  \label{fig:inlayer}
}
\caption{Surface plots of the numerical solutions.}
\end{figure}

\appendix
\section{Implementation issues for the scalar convection diffusion problems \eqref{equ:grad-conservative} with \eqref{eq:P2-sequence}}
\label{sec:implementation}
We now discuss the computation of the discrete flux in \eqref{eq:SAFE-a} including the limiting case when the
diffusion coefficient approaches zero.

\subsection{Local basis functions of $\mathcal{P}_2\Lambda^0$ and
$\mathcal{P}_1\Lambda^1$}
Given $T\in \mathcal{T}_h$ with vertices $[q_0, q_1, q_2]$, the degrees of freedom (the set of unisolvent functionals) corresponting to the $\mathcal{P}_2\Lambda^0(T)$ Lagrange
element are the function evaluations at $q_i~(i=1,2,3)$ and the integral averages on edges $f_{ij}$, where $f_{ij}$ represents the edge connecting two vertices $q_i$ and $q_j$.
The basis functions dual this set of degrees of freedom are
\begin{equation} \label{eq:basis-Lagrange}
\varphi_i = \lambda_i (3\lambda_i - 2), \quad \varphi_{ij} =
6\lambda_i\lambda_j.
\end{equation}
The space $\mathcal{P}_1\Lambda^1$, as a subspace of $H({\rm curl})$,
corresponds to the $\frac{\pi}{2}$-rotation of the well known BDM space. We introduce the tangential vectors $t_{ij} = q_j - q_i$ and $\tau_{ij} = \frac{t_{ij}}{|t_{ij}|}$.  The
set of degrees of freedom determining a function $v \in
\mathcal{P}_1\Lambda^1(T)$ then is
$$
\int_{f_{ij}} (v \cdot \tau_{ij}) p, \quad \forall p \in
\mathcal{P}_{1}(f_{ij}).
$$
Taking a basis of $\mathcal{P}_1(f_{ij})$ as $3\lambda_i -2$ and
$3\lambda_j - 2$, the basis function associated with the degrees of freedom on $f_{ij}$ are given as
\begin{equation} \label{eq:basis-BDM}
\psi_{ij}^{(1)} = 2\lambda_j \nabla\lambda_i,
  \quad
\psi_{ij}^{(2)} = -2\lambda_i \nabla\lambda_j.
\end{equation}
We note that these basis functions satisfy $\psi_{ji}^{(1)} = - \psi_{ij}^{(2)},
\psi_{ji}^{(2)} = -\psi_{ij}^{(1)}$.

\subsection{Local representation of the discrete flux}
We now recall that the definition of the discrete flux in
\eqref{eq:numerical-flux} which, for the scalar
convection-diffusion problem, is $\alpha J_{\bar{\theta},T} = \alpha (\Pi_T^1 \mathrm{e}^{\bar{\theta}\cdot x})^{-1} \nabla \Pi_T^0
\mathrm{e}^{\bar{\theta}\cdot x}$.
As we pointed out earlier, for the sake of simplicity, we consider the piecewise constant diffusion coefficient $\alpha$, which implies that $\bar{\theta} =
\bar{\beta} / \alpha$. Then, a straightforward  calculation shows that
$$
\begin{aligned}
\nabla \Pi_T^0 \mathrm{e}^{\bar{\theta} \cdot x} \varphi_i &=
\mathrm{e}^{\bar{\theta} \cdot q_i} \nabla \varphi_i +
\sum_{j\neq i} \left( \dashint_{f_{ij}} \lambda_i(3\lambda_i - 2)
\mathrm{e}^{\bar{\theta} \cdot x} \right) \nabla \varphi_{ij} \\
&= \mathrm{e}^{\bar{\theta} \cdot q_i} \sum_{j\neq i} \left[
(-\psi_{ij}^{(1)} +
2\psi_{ij}^{(2)}) + 3 \left( \dashint_{f_{ij}} \lambda_i(3\lambda_i - 2)
\mathrm{e}^{\bar{\theta} \cdot (x-q_i)}\right)(\psi_{ij}^{(1)} -
\psi_{ij}^{(2)}) \right] \\
&= \mathrm{e}^{\bar{\theta} \cdot q_i} \sum_{j \neq i}
  \left[(3V(\bar{\theta}\cdot t_{ij}) - 1) \psi_{ij}^{(1)} +
  (3V(\bar{\theta}\cdot t_{ij}) + 2) \psi_{ij}^{(2)}\right].
\end{aligned}
$$
After mapping to the unit interval, we can compute the integrals to obtain that
\begin{equation} \label{eq:vertex-exp}
V(s) :=  \int_{0}^1 (1-x)(1-3x) \mathrm{e}^{sx} \,\rd x = -
  \frac{2s\mathrm{e}^{s} - 6\mathrm{e}^{s} + s^2 + 4s+ 6}{s^3}.
\end{equation}
Next, for the basis functions of $\mathcal{P}_2\Lambda^0(T)$ associated with the edge average, we have
$$
\begin{aligned}
\nabla \Pi^0 \mathrm{e}^{\bar{\theta} \cdot x} \varphi_{ij} &=
  \dashint_{f_{ij}} 6\lambda_i\lambda_j \mathrm{e}^{\bar{\theta} \cdot
  x} ds \nabla \varphi_{ij} \\
&= 3 \mathrm{e}^{\bar{\theta} \cdot q_i} \left( \dashint_{f_{ij}}
  6\lambda_i\lambda_j  \mathrm{e}^{\bar{\theta} \cdot (x-q_i)}\right)
  (\psi_{ij}^{(1)} - \psi_{ij}^{(2)}) \\
&= \mathrm{e}^{\bar{\theta} \cdot q_i} \left[ 3E(\bar{\theta}\cdot
  t_{ij}) \psi_{ij}^{(1)} - 3E(\bar{\theta}\cdot t_{ij}) \psi_{ij}^{(2)}
  \right],
\end{aligned}
$$
where
\begin{equation} \label{eq:edge-exp}
E(s) :=  \int_0^1 6x(1-x) \mathrm{e}^{sx} \,\rd x =
  \frac{6s\mathrm{e}^s - 12 \mathrm{e}^s + 6s + 12}{s^3}.
\end{equation}

For the $1$-forms, we consider the exponential fitting corresponding to the basis functions
$\mathcal{P}_1\Lambda^1(T)$. We have,
$$
\begin{aligned}
\Pi^1 \mathrm{e}^{\bar{\theta}\cdot x}
\begin{bmatrix}
\psi_{ij}^{(1)} \\
\psi_{ij}^{(2)}
\end{bmatrix} & =
\begin{bmatrix}
 \dashint_{f_{ij}} -2\lambda_j (3\lambda_i - 2) \mathrm{e}^{\bar{\theta}\cdot x} &
 \dashint_{f_{ij}} -2\lambda_j (3\lambda_j - 2) \mathrm{e}^{\bar{\theta}\cdot x} \\
 \dashint_{f_{ij}} -2\lambda_i (3\lambda_i - 2) \mathrm{e}^{\bar{\theta}\cdot x} &
 \dashint_{f_{ij}} -2\lambda_i (3\lambda_j - 2) \mathrm{e}^{\bar{\theta}\cdot x}
\end{bmatrix}
\begin{bmatrix}
\psi_{ij}^{(1)} \\
\psi_{ij}^{(2)}
\end{bmatrix} \\
&= \mathrm{e}^{\bar{\theta} \cdot q_i}
\begin{bmatrix}
a_{11}(\bar{\theta}\cdot t_{ij}) &
a_{12}(\bar{\theta}\cdot t_{ij}) \\
a_{21}(\bar{\theta}\cdot t_{ij}) &
a_{22}(\bar{\theta}\cdot t_{ij})
\end{bmatrix}
\begin{bmatrix}
\psi_{ij}^{(1)} \\
\psi_{ij}^{(2)}
\end{bmatrix}
:= \mathrm{e}^{\bar{\theta} \cdot q_i}  \bm{A}(\bar{\theta}\cdot t_{ij})
\begin{bmatrix}
\psi_{ij}^{(1)} \\
\psi_{ij}^{(2)}
\end{bmatrix},
\end{aligned}
$$
where $\bm{A}(s)$ is the matrix with entries defined by
\begin{equation}\label{eq:BDM-exp}
\begin{aligned}
a_{11}(s) &:= \int_0^1 -2 x (1 - 3x) \mathrm{e}^{sx}\,\rd x =
  \frac{4s^2\mathrm{e}^s - 10s\mathrm{e}^s + 12\mathrm{e}^s - 2s -
  12}{s^3}, \\
a_{12}(s) &:= \int_0^1 -2 x (3x - 2) \mathrm{e}^{sx}\,\rd x =
  -\frac{2s^2 \mathrm{e}^s - 8s \mathrm{e}^s + 12\mathrm{e}^s - 4s -
  12}{s^3}, \\
a_{21}(s) &:= \int_0^1 -2(1-x)(1-3x) \mathrm{e}^{sx}\,\rd x = \frac{4s
  \mathrm{e}^s - 12\mathrm{e}^s + 2s^2 + 8s + 12}{s^3}, \\
a_{22}(s) &:= \int_0^1 -2(1-x)(3x-2) \mathrm{e}^{sx}\,\rd x =
  -\frac{2s \mathrm{e}^s - 12 \mathrm{e}^s + 4s^2 + 10s + 12}{s^3}.
\end{aligned}
\end{equation}

Using the calculation above, and recalling that $\bar{\theta} = \bar{\beta} / \alpha$ for the piecewise constant diffusion coefficient, we arrive at the following representation of the numerical flux
\begin{equation} \label{eq:flux-EAFE}
\begin{aligned}
\alpha J_{\bar{\theta}, T} \varphi_i &= \sum_{j \neq i} \left[
  B_{V,1}^\alpha(\bar{\beta}\cdot t_{ij}) \psi_{ij}^{(1)} +
  B_{V,2}^\alpha(\bar{\beta}\cdot t_{ij}) \psi_{ij}^{(2)}\right],\\
\alpha J_{\bar{\theta}, T}\varphi_{ij} &= B_{E,1}^\alpha(\bar{\beta}
  \cdot t_{ij}) \psi_{ij}^{(1)} + B_{E,2}^\alpha(\bar{\beta} \cdot t_{ij})
  \psi_{ij}^{(2)},
\end{aligned}
\end{equation}
where, the {\it generalized Bernoulli functions} are defined via
$V(\cdot)$ in \eqref{eq:vertex-exp}, $E(\cdot)$ in \eqref{eq:edge-exp},
and $\bm{A}(\cdot)$ in \eqref{eq:BDM-exp} as follows
\begin{equation} \label{eq:Bernoulli}
\begin{aligned}
(B_{V,1}^\alpha(s), B_{V,2}^\alpha(s)) &:= \alpha (3V(s/\alpha)-1,
  3V(s/\alpha)+2) \left[\bm{A}(s/\alpha)\right]^{-1}, \\
(B_{E,1}^\alpha(s), B_{E,2}^\alpha(s))&: = \alpha (3E(s/\alpha),
  -3E(s/\alpha)) \left[\bm{A}(s/\alpha)\right]^{-1}.
\end{aligned}
\end{equation}
As a consequence, the formula for the discrete flux \eqref{eq:flux-EAFE}, provides the practical method for calculating the local stiffness matrix
$(\alpha J_{\bar{\theta}, T} u_h, \nabla v_h)_T$.

\subsection{Limiting case for vanising diffusion coefficient}
When $s/\alpha$ is a large quantity
the some exponential functions in the definition of
$V(\cdot)$, $E(\cdot)$ and $\bm{A}(\cdot)$ might be difficult to compute directly. We now provide formulae which can be used in the limiting case. As a side result, these formulae give the behaviour of the numerical scheme in case of vanishing diffusion.
In summary, we will compute the limits of generalized Bernoulli functions in \eqref{eq:Bernoulli} as the
diffusion coefficient approaches zero, and, in addition we shall show the
consistency of the bilinear form \eqref{eq:SAFE-a} to the one corresponding to the pure diffusion case as the convection vanishes.
Such calculations are based on the following lemma, whose proof is elementary and omitted here.

\begin{lemma} \label{lm:Bernoulli-limit}
The generalized Bernoulli functions \eqref{eq:Bernoulli} have the
properties:
\begin{enumerate}
\item For a fixed $\alpha$, it holds that
\begin{equation} \label{eq:limit-s0}
\begin{aligned}
\lim_{s \to 0}(B_{V,1}^\alpha(s), B_{V,2}^{\alpha}(s)) &= (-\alpha,
  2\alpha), \\
\lim_{s \to 0}(B_{E,1}^\alpha(s), B_{E,2}^{\alpha}(s)) &= (3\alpha,
  -3\alpha).
\end{aligned}
\end{equation}
\item For a fixed $s$, it holds that
\begin{equation} \label{eq:limit-alpha0}
\begin{aligned}
\lim_{\alpha \to 0^+} (B_{V,1}^\alpha(s), B_{V,2}^{\alpha}(s))  &=
\begin{cases}
(0, s) & s > 0, \\
(\frac{3s}{2}, -\frac{s}{2}) & s < 0.
\end{cases} \\
\lim_{\alpha \to 0^+} (B_{E,1}^\alpha(s), B_{E,2}^{\alpha}(s)) &=
\begin{cases}
(0, -3s) & s > 0, \\
(-3s, 0) &  s< 0.
\end{cases}
\end{aligned}
\end{equation}
\end{enumerate}
\end{lemma}
The limits  in Lemma~\ref{lm:Bernoulli-limit} lead to the following conclusions:
Firstly, there exists a stable implementation of the generalized
Bernoulli functions by using the limiting values for large $s$. Secondly, the  discrete flux $\alpha J_{\bar{\theta},T}$
\eqref{eq:flux-EAFE} and the resulting exponentially fitted discretization are well defined when the diffusion coefficient approaches zero.

\section*{Acknowledgements}
Part of this work was completed when the second author was visiting the School of Mathematical Sciences at Peking University and we thank Prof. Jun Hu (PKU) and Prof. Chensong Zhang (LSEC) for the discussions during the preparation of this manuscript.

The work of Shuonan Wu is supported in part by the National Natural Science Foundation of China grant No.~11901016 and the startup grant from Peking University. The work of Zikatanov is supported in part by the US National Science Foundation awards DMS-1720114 and DMS-1819157.

\bibliographystyle{amsplain}
\bibliography{bib_eafe}
\end{document}